\documentclass[11pt]{article}

\RequirePackage{amsthm,amssymb}

\usepackage{multirow}
\usepackage{amsmath}
\usepackage[hidelinks]{hyperref}
\usepackage[letterpaper, margin=1in]{geometry}
\newcommand{\om}{\omega}
\newcommand{\oequiv}{\approx}
\newcommand{\setconcat}{+}
\newcommand{\COL}{\operatorname{COL}}
\newcommand{\N}{\mathbb{N}}
\newcommand{\Z}{\mathbb{Z}}

\newcommand{\capcrules}{Coloring Rules}
\newcommand{\crulenospace}{CR}
\newcommand{\cruleLong}{coloring rule}
\newcommand{\crulesLong}{coloring rules}
\newcommand{\crule}{\crulenospace~}
\newcommand{\crules}{\crulenospace s~}
\newcommand{\crulesnospace}{\crulenospace s}

\newcommand{\capgencrules}{General Coloring Rules}
\newcommand{\gencrulenospace}{GCR}
\newcommand{\gencruleLong}{general coloring rule}
\newcommand{\gencrulesLong}{general coloring rules}
\newcommand{\gencrule}{\gencrulenospace~}
\newcommand{\gencrules}{\gencrulenospace s~}
\newcommand{\gencrulesnospace}{\gencrulenospace s}

\newcommand{\xleq}{\preceq_{\overline{A}}}
\newcommand{\assny}{\overline{B}}
\newcommand{\assnc}{\overline{C}}

\newcommand{\case}[1]{\textbf{Case #1.}}

\newcommand{\numpoints}{n}
\newcommand{\inumpoints}{i}
\newcommand{\idegr}{j}
\newcommand{\numcolors}{c}
\newcommand{\commonconst}{a}
\newcommand{\Commonconst}{A}
\newcommand{\rareconst}{b}
\newcommand{\Rareconst}{B}

\theoremstyle{plain}
\newtheorem{theorem}{Theorem}[section]
\newtheorem{lemma}[theorem]{Lemma}

\theoremstyle{definition}
\newtheorem{definition}[theorem]{Definition}

\usepackage[utf8]{inputenc}

\title{Big Ramsey Degrees of Countable Ordinals}

\author{Joanna Boyland\thanks{Supported by NSF grant 2150382.}\\
\small Department of Philosophy\\[-0.8ex]
\small Carnegie Mellon University \\[-0.8ex] 
\small Pittsburgh, PA, 15213 \\
\small\tt joannabo@andrew.cmu.edu
\and
William Gasarch\footnotemark[1]\\
\small Department of Computer Science\\[-0.8ex]
\small University of Maryland at College Park\\[-0.8ex]
\small College Park, MD, 20742\\
\small\tt gasarch@umd.edu
\and
Nathan Hurtig\footnotemark[1]\\
\small School of Computer Science \& Engineering\\[-0.8ex]
\small University of Washington\\[-0.8ex]
\small Seattle, WA, 98195\\
\small\tt hurtig@cs.washington.edu
\and
Robert Rust\footnotemark[1]\\
\small Department of Mathematics and Statistics\\[-0.8ex]
\small Binghamton University\\[-0.8ex]
\small Binghamton, NY, 13902\\
\small\tt rrust@binghamton.edu}

\begin{document}

\maketitle

\begin{abstract}


Ramsey's theorem on countable infinite sets states that for all natural numbers $n,$ for all finite colorings of the $n$-element subsets of some infinite countable set, there exists an infinite countable homogeneous subset. What if we seek a homogeneous subset that is also order-equivalent to the original set?
Let $S$ be a linearly ordered set and $\numpoints \in\N$.  
The big Ramsey degree of $\numpoints$ in $S$, denoted $T(\numpoints,S)$,  is the least natural number $t$ such that, for any finite coloring of the size $\numpoints$ subsets of $S$, there exists $S'\subseteq S$ such that (i) $S'$ is order-equivalent to $S$, and (ii) if the coloring is restricted to the size $\numpoints$ subsets of $S'$ then at most $t$ colors are used. 

Ma\v{s}ulovi\'{c} \& \v{S}obot (2021) showed that
$T(\numpoints,\omega+\omega)=2^{\numpoints}$. From this one can obtain $T(\numpoints,\zeta)=2^{\numpoints},$ where
$\zeta$ is the ordered set of integers.
We give a direct proof that $T(\numpoints,\zeta)=2^{\numpoints}$. 

Ma\v{s}ulovi\'{c} and \v{S}obot (2021) also showed that for all
countable ordinals $\alpha < \om^\om$ and $\numpoints \in \N$, $T(\numpoints,\alpha)$  is finite. 
We find exact values of $T(\numpoints,\alpha)$ for all ordinals $\alpha$ less than $\om^\om$ and all $\numpoints \in\N$.

\end{abstract}

\section{Introduction}

\begin{definition}
Let $\N$ be the set of natural numbers, including 0. 
Let $\numpoints,\numcolors \in \N$ and $S$ be a set. Then we define $[\numcolors]$ to be $ \{ 0, \ldots, \numcolors-1\}$
(note $0 \in [\numcolors]$, but $\numcolors \not \in [\numcolors]$). Note that if $\numcolors=0$ then $[\numcolors] = \emptyset$. 
We use $\binom{S}{\numpoints}$ to denote the set of all $\numpoints$-element subsets of $S$. 
\end{definition}
\begin{definition}
    A function $\COL\colon \binom{S}{\numpoints} \to [\numcolors]$ is called a \textit{coloring}.
    A set $H\subseteq S$ is called \textit{homogeneous} if $\COL$ restricted to 
    $\binom{H}{\numpoints}$ is constant,
    that is, $\left|\COL\left(\binom{H}{\numpoints}\right)\right|\le 1$.
\end{definition}

Ramsey's Theorem on infinite countable sets \cite{Ramsey-1930} is as follows:
\begin{theorem} 
Let $\numpoints,\numcolors\in\N,$ and let $S$ be an infinite countable set.  
For all colorings $\COL\colon \binom{S}{\numpoints}\to [\numcolors]$ 
there exists an infinite homogeneous set $H\subseteq S$. 
\end{theorem} 

If $S$ is an \textit{ordered} set, then Ramsey's Theorem does not guarantee a
homogeneous set of the same order type. Instead, we might be able to get a set $H\subseteq S$ of the same order type where $\left|\COL\left(\binom{H}{\numpoints}\right)\right|$ is bounded by some constant. Ma\v{s}ulovi\'{c} and \v{S}obot~\cite{MS-2021} prove the following:

\begin{theorem} 
Let $\numpoints,\numcolors\in\N,$ and let $S$ be a set
of order type $\zeta$ (the integers).
For all colorings $\COL\colon \binom{S}{\numpoints}\to [\numcolors]$ 
there exists an infinite homogeneous set $H\subseteq S$ such that 
$\left|\COL\left(\binom{H}{\numpoints}\right)\right|\le 2^n$. Note that $2^n$ is independent of 
$\numcolors$. 
\end{theorem} 

Our main results are about the colorings of finite subsets of countable ordinals less than $\omega^\omega$. 
We provide necessary definitions in Section~\ref{se:defs},
and summarize our results rigorously in Section~\ref{se:summary}. 

\section{Definitions}\label{se:defs}

\begin{definition} 
Let $\mathcal{A} = (A,\preceq_A)$ and $\mathcal{B} = (B,\preceq_B)$ be ordered sets, where
$\preceq_A, \preceq_B$ are total orders on $A, B$ respectively.

\begin{enumerate}

\item Two ordered sets
$\mathcal{A},\mathcal{B}$ are \textit{order-equivalent}, denoted $\mathcal{A}\oequiv \mathcal{B}$, 
if there exists an order-preserving bijection $f\colon A \to B$ such that, for all $a_1, a_2 \in A$: 

$$a_1 \preceq_A a_2 \iff f(a_1) \preceq_B f(a_2).$$

\item 
The ordered set $\mathcal{A} \setconcat \mathcal{B}$ is defined to be
$(A\sqcup B, \preceq),$ where the disjoint union $A\sqcup B$ is formally defined to be
$(\{0\} \times A) \, \cup\, (\{1\} \times B)$, and where $\preceq$ is the dictionary ordering on $A\sqcup B.$
Note that $\setconcat$ agrees with the definition of ordinal addition, and is not commutative in general.

\item 
Let $\mathcal{A}=(A,\preceq_A)$ be an ordered set with $A \subseteq \Z$. Then define $-\mathcal{A}$ to be $(-A,\preceq_{-A})$, where 
$-A = \{ -a \colon a \in A\}$, and
$-a \preceq_{-A} -b \iff b \preceq_{A} a$.
\end{enumerate} 

\end{definition} 

Throughout this paper, we conflate the notation for an ordered set and its underlying set. 
For example, if $A$ is an ordered set, we might still use the notation $\binom{A}{2}$.

\begin{definition}\label{def:homogeneous}
Let $S$ be an ordered set, $S'\subseteq S$, and $\numpoints,\numcolors,t \in \N.$ Let $\COL \colon \binom{S}{\numpoints} \to [\numcolors]$ be a coloring.
\begin{enumerate}
    \item 
$S'$ is \textit{homogeneous} if  $\left|\COL\left(\binom{S'}{\numpoints}\right)\right|= 1$. 
\item 
 $S'$ is \textit{$t$-homogeneous} if $\left|\COL\left(\binom{S'}{\numpoints}\right)\right|\leq t$. 
 \item 
$S'$ is $S$-\textit{$t$-homogeneous} if both $S'$ is $t$-homogeneous and $S'\oequiv S$.
\end{enumerate}
\end{definition}

\begin{definition}
We identify every ordinal $\alpha$ as the ordered set of all ordinals $\beta < \alpha$, with 0 as the least ordinal. Additionally, let $\omega$ be the ordered set containing only the naturals, $\zeta$ be the ordered set containing only the integers, and $\eta$ be the ordered set containing only the rationals, all under their respective natural orderings.

\end{definition}

\begin{definition}\label{def:T}
Let $S$ be an ordered set.
For any $\numpoints\in \N$, $T(\numpoints,S)$
is the least $t\in \N$ such that for all constants $\numcolors \geq 1$ and
colorings  $\COL \colon \binom{S}{\numpoints} \to [\numcolors]$, there exists some $S' \subseteq S$ 
such that $S'$ is $S$-$t$-homogeneous.
Note that $t$ is independent of $\numcolors$.
If no such $t$ exists then we define $T(\numpoints,S)=\infty$. 
The quantity $T(\numpoints,S)$ is called the \textit{big Ramsey degree} of $\numpoints$ in $S$.
The term was first coined by Kechris et al.~\cite{KPT-2005}. Other sources use the Ramsey arrow notation modified for ordered sets. In that notation, $T(\numpoints,S)=t$ is short-hand for both
$S \to (S)^{\numpoints}_{\numcolors,t}$ and $S \not \to
(S)^{\numpoints}_{T(\numpoints,S),t-1}$ for all $\numcolors \in \N.$
We use the $T(n, S)$ notation in this paper.
\end{definition}

The following stems from Definition~\ref{def:T}, but we include it for completeness.
\begin{lemma}\label{lem:choose0}
    For all ordered sets $S,$ $T(0,S) = 1.$ 
\end{lemma}

Ramsey's Theorem on $\N$ gives an infinite 1-homogeneous subset of $\N$. Theorem~\ref{th:ramsey} restates Ramsey's Theorem using big Ramsey degree notation.

\begin{theorem}\label{th:ramsey}
For all $\numpoints \in \N,$
$T(\numpoints,\omega)=1.$
\end{theorem}







\section{Related Work and our Work} \label{se:summary}

We summarize what is known about the big Ramsey degrees of linear orders,
and what our contribution is.
We also briefly discuss big Ramsey degrees for other structures.
We then give a more detailed account of our contributions.
All of the structures in this section are countable.

\subsection{Big Ramsey Degrees of Linear Ordering and Other Structures}

The table below summarizes what is known about the big Ramsey degrees of
ordinals. We use $\beta$ to denote an arbitrary ordinal and $n$ for an arbitrary positive number.

\begin{center}
\begin{tabular}{|c|ll|} 
\hline
Ordinal & Known Results & Comments \cr
\hline
$\om$ & $T(n,\om)=1$~\cite{Ramsey-1930} & Infinite Ramsey Theorem \cr 
$\om^\beta$ & $T(1,\om^\beta)=1$~\cite{Frasse-2000} (page 189) & \cr
$\beta$ & $T(1,\beta)<\infty$~\cite{Frasse-2000} (page 189) & \cr
$\beta<\om^\om$ & $T(n,\beta)<\infty$~\cite{MS-2021} & We provide exact numbers \cr
$\beta\geq\om^\om$ & $(\forall n\ge 2)[T(n,\beta)=\infty]$~\cite{MS-2021} &  \cr
\hline
\end{tabular}
\end{center}

\begin{definition}
Let $L$ be a countable scattered linear ordering.
Then ${\rm HR}(L)$ denotes the Hausdorff rank of $L$. 

\end{definition}

In the table below,
the $i$th Bernoulli number is denoted by $B_i$.
We use $L$ to denote an arbitrary countable  scattered linear ordering.
\begin{center}
\begin{tabular}{|c|ll|} 
\hline
Ordering & Known Results & Comments \cr
\hline
$\zeta$  & $T(n,\zeta)=2^n$~\cite{MS-2021} & We provide an easier proof \cr 
$\eta$   & $T(n,\eta) = \frac{B_{2n+1}(-1)^{n+1}(1-4^{n+1})}{(2(n+1))!}$~\cite{Devlin-1979,Vuk-2002}  & \cr
$L$ & $T(1,L)<\infty$~\cite{Laver-1973} & \cr 
$L$ & $(\forall n \ge 2)[ {\rm HR}(L)=\infty \iff T(n,L)=\infty]$~\cite{BMN-2023,Masulovic-2023} & \cr 
\hline
\end{tabular}
\end{center} 

Big Ramsey degrees have also been studied for structures more complex than
linear orders:
\begin{enumerate} 
\item 
the generic partial order~\cite{BCDHKVZ-2023,Hubicka-2020},
\item 
the Rado graph~\cite{EHP-1975,LSV-2006,Larson-2008,PS-1996,Sauer-2006},
\item 
the Rado 
hypergraph~\cite{BCHKV-2020,BCRHKK-2023,CDP-2022-ind,CDP-2022-simply}, and
\item 
the Henson graph~\cite{Dobrinen-2020-I,Dobrinen-2020-II,Dobrinen-2023}.
\end{enumerate} 

More abstract versions of these questions have also been 
studied~\cite{Masulovic-2020,The-2009,Zucker-2022}.
For a survey and more references, see Dobrinen~\cite{Dobrinen-2022} and 
Coulson et al.~\cite{CDP-2022-ind,CDP-2022-simply}.
Page 52 of~\cite{CDP-2022-simply} has a table of known results.
For a collection of open problems, see
Dobrinen and Gasarch~\cite{dobrinen-open-2020}.

Exact values are known for
coloring edges of the Rado graph:
upper bounds by
Pouzet \& Sauer~\cite{PS-1996} and
lower bounds by 
Erd\H{o}s, Hajnal and Posa~\cite{EHP-1975}. 
When coloring arbitrary sets of $n$ vertices in
the Rado graph,
exact values are known: see Leflamme, Sauer, and
Vuksanovic~\cite{LSV-2006} and Larson~\cite{Larson-2008}.
For all $a\ge 3$, the Rado $a$-hypergraph has finite big Ramsey degree;
however, no exact numbers are known. 
For all $k$, the $k$-Henson graph has finite big Ramsey degree. The exact big Ramsey degrees of 3-Henson graphs are known, see Dobrinen~\cite[Corollary 5.13]{Dobrinen-2020-II}. Dobrinen has made a conjecture about the Ramsey degree of $k$-Henson graphs~\cite[Section 11: Future Directions]{Dobrinen-2023},
and also supplies some numerical results. 

\subsection{Summary of Results}


Our work is most similar to that of Ma\v{s}ulovi\'{c} \& \v{S}obot~\cite{MS-2021},
who study the conditions under which $T(n,\beta)$ is finite. Their methods
rely on what they call a top-down approach, where they prove that the
property of finite big Ramsey degrees is passed from ordinals of the form $\om^n$
to their subsums of lesser ordinals. Our methods are bottom-up: we provide explicit formulas
for calculating the degree of an ordinal number from its subsums. This enables
us to compute exact results, unlike the work of Ma\v{s}ulovi\'{c} \& \v{S}obot.

In other words, ordinal
numbers less than $\om^\om$ can be viewed as summations of smaller
and simpler ordinals of the form $\om^n$. Our main result relies on the
big Ramsey degrees of these simpler ordinals to determine the big Ramsey
degrees of all ordinals less than $\om^\om.$

In this paper, we do the following.

\begin{enumerate}
\item
We show that $T(\numpoints,\zeta)=2^{\numpoints}$ in Section~\ref{se:Z}. This can be obtained
by the result due to Ma\v{s}ulovi\'{c} \& \v{S}obot~\cite{MS-2021} that 
$T(\numpoints,\omega+\omega)=2^\numpoints$, but we give a direct proof using different methods.
\item
We find the big Ramsey degrees $T(n, \om \cdot k)$ for finite $k$ in Section~\ref{se:aomega}.
\item
We find the big Ramsey degree $T(2, \om^2)$ in Section~\ref{se:omega2} 
and introduce our main method for finding big Ramsey degrees of ordinals
in Section~\ref{se:strong}.
\item We find the big Ramsey degrees $T(n, \om^d)$ and
$T(n, \om^d \cdot k)$ for finite $d,k$ in Sections~\ref{se:ord} and~\ref{se:ord2}.
\item We find the big Ramsey degrees of all ordinals $\alpha<\omega^\omega$
in Section~\ref{se:done}.
\end{enumerate}





\section{Big Ramsey Degrees of \texorpdfstring{$\zeta$}{Z}}\label{se:Z}

Our general result is $T(n, \zeta) = 2^n$.
We first prove $T(1,\zeta)=2$
and $T(2,\zeta)=4$. Both of these will 
build notation and structure toward the main result.

\begin{theorem}
The big Ramsey degree $T(1,\zeta)=2$. 
\end{theorem}

\begin{proof}
We first prove $T(1,\zeta)\leq 2$. Let $\numcolors \in \N$ and let $$\COL\colon \zeta\to [\numcolors]$$ be a 
$\numcolors$-coloring of $\zeta$.
Let $\COL'\colon \omega \to [\numcolors]\times[\numcolors]$ be defined by 

$$\COL'(x) = (\COL(-x),\COL(x)).$$

By Theorem~\ref{th:ramsey} (Ramsey's theorem), there exists a set $G \subseteq \om$
that is $\om$-1-homogeneous with respect to $\COL'$. Let the color of the homogeneous set be $(c_1, c_2)$.
Consider the set $H=-G \setconcat G$, which is order-equivalent to $\zeta$. Let $h \in H$. If $h 
\in G$ then by the definition of $\COL'$ and 1-homogeneity of $G$, $\COL(h)=c_2$. Otherwise, $h\in -G$, and so $\COL(h)=c_1$. Thus $H$ is $\zeta$-2-homogeneous. Because $\COL$ was arbitrary, $T(1,\zeta)\leq 2$.

We now prove $T(1,\zeta)\geq 2$. Let $\COL\colon\zeta\to [2]$ be the function
that maps all nonnegative integers to $0$ and all
negative integers to $1$.
Since the nonnegative integers have no infinitely descending chain and the negative integers have no infinitely ascending chain, there is no $\zeta$-1-homogeneous subset of $\zeta$ under $\COL$. Therefore $T(1,\zeta)\geq 2$, and with the previous result, $T(1,\zeta) = 2$.
\end{proof}

\begin{theorem}\label{the:Z2}
The big Ramsey degree $T(2,\zeta)=4$. 
\end{theorem}
\begin{proof}
We first prove $T(2,\zeta)\leq 4$. Let $\numcolors \in \N$ and let $$\COL\colon \binom{\zeta}{2}\to [\numcolors]$$
be a $\numcolors$-coloring of $\binom{\zeta}{2}$.
Let $\COL'\colon \binom{\omega}{2} \to [\numcolors]^4$ be defined by 

$$
\COL'(x,y) = (\COL(-x,-y), \COL(-x,y), \COL(x,-y), \COL(x,y))
$$ where $x < y.$

By Theorem~\ref{th:ramsey} there exists an $\om$-1-homogeneous set $G$. 
Let the color of the homogeneous set be $(c_1,c_2,c_3,c_4)$.
Let $G = \{h_0 < h_1 < \cdots \}$. 

Then consider the set $$H= \{-h_i \colon i \textnormal{ is even}\} \setconcat \{h_i \colon i \textnormal{ is odd}\},$$ which is order-equivalent to $\zeta$. Let $s_in_i, s_jn_j \in H,$ where $n_i,n_j\geq 0$ and $s_i,s_j \in \{-1,1\}$. Then by definition of $\COL'$, $\COL(s_in_i,s_jn_j)$ is in $\{c_1, c_2, c_3, c_4\}$, depending on $s_i$ and $s_j$. As an aside, this method only works when $n_i$ and $n_j$ are guaranteed to be distinct, which we forced with our alternation of sign by the parity of index.
Since $\COL(H)\subseteq \{c_1,c_2,c_3,c_4\}$, $H$ is $\zeta$-4-homogeneous.
Therefore $T(2,\zeta)\leq 4$.

We now prove $T(2,\zeta)\geq 4$. Let $\COL\colon\binom{\zeta}{2}\to [4]$ be the coloring
\begin{align*}
\COL(x,y)=
\begin{cases}
0 & \text{$x,y\geq 0$} \cr
1 & \text{$x\geq 0$, $y<0$, and $|x|\leq |y|$} \cr
2 & \text{$x\geq 0$, $y<0$, and $|x|> |y|$} \cr
3 & \text{$x<0$, $y<0$.} \cr
\end{cases}
\end{align*}
Let $H$ be an arbitrary order-equivalent subset of $\zeta.$
The following $x,y$ are guaranteed to exist because $H \subseteq \zeta$ and $H \oequiv \zeta$:
\begin{enumerate} 
\item
Let $x,y\in H$ such that $x,y\ge 0$. Then $\COL(x,y)=0$. 
\item
Let $x,y\in H$ such that $x\ge 0$, $y<0$, and $|x|\le |y|$. Then $\COL(x,y)=1$. 
\item 
Let $x,y\in H$ such that $x\ge 0$, $y<0$, and $|x|>|y|$. Then $\COL(x,y)=2$. 
\item
Let $x,y\in H$ such that $x<0$, $y<0$. Then $\COL(x,y)=3$. 
\end{enumerate}  
Therefore $T(2,\zeta)\geq 4$, and with the previous result, $T(2,\zeta) = 4$.
\end{proof}

The proof of Theorem~\ref{the:Z2} exposes the common structure of all the proofs in this
paper. We identify $T(n, \alpha)$ ways to reshape some naturals into specific formats, and
define a $\COL'$ using those cases. After applying Ramsey's theorem to yield some
order-equivalent subset of naturals $G,$ we rebuild some
order-equivalent $H$ whose coloring under $\COL$ is guaranteed by the $T(n, \alpha)$ cases.
Our upper bound proofs use the same cases again.

\begin{theorem}\label{th:zeta}
For all $\numpoints \in \N$, $T(\numpoints,\zeta)=2^{\numpoints}$. 
\end{theorem}

\begin{proof}
We first prove
$T(\numpoints,\zeta)\leq 2^\numpoints$. 
Let $\numcolors \in \N$ and let 
$\COL\colon \binom{\zeta}{\numpoints}\to [\numcolors]$ 
be a $\numcolors$-coloring of $\binom{\zeta}{\numpoints}$.
Let $\COL'\colon \binom{\omega}{n}\to [\numcolors]^{(2^n)}$ be defined by 
$$\COL'(x_1,\ldots,x_{\numpoints}) = (\COL(x_1,\ldots,x_{\numpoints}),\COL(-x_1,x_2,\ldots,x_{\numpoints}),\ldots,\COL(-x_1,\ldots,-x_{\numpoints}))$$ where $x_1 < x_2 < \cdots < x_{\numpoints}.$
Formally, $\COL'$ outputs a tuple with $\COL$'s output on each item in the set $$
\{-x_1,x_1\} \times \{-x_2,x_2\} \times \cdots \times \{-x_\numpoints, x_\numpoints\}.
$$
Note that $\COL'$ only depends on the color of elements of $\binom{\zeta}{\numpoints}$ where the absolute values of the integers are all different. 

By Theorem~\ref{th:ramsey} there exists some homogeneous set $G$. Let $G = \{h_0 < h_1 < \cdots \}$.
Then the set $$H= \{-h_i \colon i \textnormal{ is even}\} + \{h_i \colon i \textnormal{ is odd}\}$$ is $\zeta$-$2^{\numpoints}$-homogeneous. Because the absolute values of every element in $H$ are all different, each subset with size $\numpoints$ was considered by $\COL'$. Therefore $T(\numpoints,\zeta)\leq 2^{\numpoints}$.

We now prove $T(\numpoints,\zeta)\geq 2^{\numpoints}$.
We describe a coloring $\COL \colon \binom{\zeta}{\numpoints} \to [2]^{\numpoints}$.

Let $\{\commonconst_1<\cdots<\commonconst_{\numpoints}\} \in \binom{\zeta}{\numpoints}$.
Define an ordering $<^\ast$ as $x<^\ast y$ if $|x| < |y|$ or both $|x|=|y|$ and $x<y$ (we order by absolute values, and in case of ties, order the negative before the positive).

Let $(i_1,\ldots,i_{\numpoints})$ be such that 

$$\commonconst_{i_1} <^\ast  \commonconst_{i_2} <^\ast  \cdots <^\ast  \commonconst_{i_{\numpoints}}.$$

Let $s_{i_j}$ be $0$ if $\commonconst_{i_j}<0$ and $1$ if $\commonconst_{i_j}\geq 0$. 
We define $\COL(\{\commonconst_1<\cdots<\commonconst_{\numpoints}\})$ as $(s_{i_1},\ldots,s_{i_{\numpoints}})$. 
We use $2^{\numpoints}$ colors, as there are $\numpoints$ elements colored either $0$ or $1$ in the output
of $\COL.$
We leave it to the reader to show that there is
no $\zeta$-$(2^\numpoints-1)$-homogeneous set. 

\end{proof} 

\section{Big Ramsey Degrees of Finite Multiples of \texorpdfstring{$\om$}{omega}}\label{se:aomega} 

As noted in Theorem~\ref{th:ramsey}, $T(\numpoints,\omega)=1$.
In this and later sections we examine limit ordinals larger than $\omega$.
For simplicity in stating results, we save the big
Ramsey degrees of successor ordinals such as $\om + 1, \om+2, \ldots$ as well as some limit ordinals such as $\om^2 + \om$ for 
Section~\ref{se:done}.

Our first result shows that $T(\numpoints, \omega \cdot k) = k^{\numpoints}$ for all $\numpoints,k$ where at least one is nonzero. 
Lemma~\ref{lem:choose0} shows that 
$T(0, \omega \cdot 0) = T(0, \emptyset) = 1$.

\begin{theorem}\label{th:omka-lt}
For $\numpoints,k \in \N$ with at least one of $\numpoints,k$ nonzero,~$T(\numpoints,\om \cdot k) \leq k^{\numpoints}$.
\end{theorem}
\begin{proof}
Let $\numpoints,\numcolors, k  \in \N$ and let 
$$\COL \colon \binom{\om \cdot k}{\numpoints} \to [\numcolors]$$
be a $\numcolors$-coloring of $\binom{\om \cdot k}{\numpoints}$.
Let $\COL' \colon \binom{\om}{\numpoints} \to [\numcolors]^{(k^{\numpoints})}$ 
be defined by
\begin{align*}\COL'(x_1,x_2,\ldots,x_{\numpoints}) & = (\COL(x_1,\ldots,x_{\numpoints}), \COL(\om + x_1,x_2,\ldots,x_{\numpoints}), \ldots, \\&~~~~~\COL(\om \cdot (k-1) + x_1, \ldots, \om \cdot (k-1) + x_{\numpoints}))\end{align*}
with $x_1 < x_2 < \cdots < x_{\numpoints}.$ Our constructed coloring 
$\COL'$ maps $\numpoints$ elements of $\om$ to the $\COL$ of each of the $k^{\numpoints}$ ways to add one of $\om\cdot 0$ through $\om \cdot (k-1)$ with each of the $\numpoints$ coordinates. 
Formally, its output tuple contains $\COL$'s assignment of each element of $$
\{x_1,\om+x_1,\ldots,\om\cdot(k-1)+x_1\}\times\cdots\times
\{x_{\numpoints},\om+x_\numpoints,\ldots,\om\cdot(k-1)+x_\numpoints\}.
$$

Apply Theorem~\ref{th:ramsey} with $\COL'$ to find some $G\oequiv \om$ such that 
$$
\left |\COL'\left(\binom{G}{\numpoints}\right)\right|=1.
$$ 
Let the one color in $\COL'\left(\binom{G}{\numpoints}\right)$ be $Y$. 
Note that $Y$ is a tuple of length $k^{\numpoints}$. 

Enumerate $G$ in order as $\{g_0 < g_1 < \cdots\}$ and let \begin{align*}
H = & \{\om \cdot 0 + g_i \colon i \equiv 0 \mod k\}~\!\setconcat\\
 & \{\om \cdot 1 + g_i \colon i\equiv 1 \mod k \} \setconcat \cdots \setconcat  \\
& \{\om \cdot (k-1) + g_i \colon i \equiv k-1 \mod k\}.
\end{align*} For the case of $k=3$, we have \begin{align*}
    H = \{ & ~~~~~~~~~g_0, ~~~~~~~~~g_3, ~~~~~~~~~~g_6, \ldots,\\
    & ~~~~\!\om + g_1, ~~~~\!\om + g_4, ~~~~\om + g_7, \ldots,\\
    & \om \cdot 2 + g_2, \om \cdot 2 + g_5, \om \cdot 2 + g_8,\ldots\}.
\end{align*} Now $H \oequiv \om \cdot k$. Then
$$
\left |\COL\left(\binom{H}{\numpoints}\right)\right |\leq k^{\numpoints}:
$$ 
for any selection of $\numpoints$ elements from $H$, its color was considered in $\COL'$ so it must be one of the $k^{\numpoints}$ colors in $Y$.
\end{proof}

We now prove that $T(\numpoints, \omega \cdot k)$ is bounded below by $k^{\numpoints}$ by providing an example coloring. The coloring is inspired by Theorem~\ref{th:omka-lt}. This duality will be made more clear in later sections.

\begin{theorem}\label{th:omka-gt}
For $\numpoints,k \in \N$,~$T(\numpoints,\om \cdot k) \geq k^{\numpoints}$. Therefore, by Theorem~\ref{th:omka-lt}, $T(\numpoints,\om \cdot k) = k^{\numpoints}$.
\end{theorem}
\begin{proof}
We give a $k^{\numpoints}$-coloring of $\binom{\om \cdot k}{\numpoints}$ that has no ($k^{\numpoints}-1$)-homogeneous $H\oequiv \om\cdot k$.
We represent $\om \cdot k$ as
$$\om \cdot k \oequiv \Commonconst_1 \setconcat \cdots \setconcat \Commonconst_k$$
where each $\Commonconst_i \oequiv \om$.
If an element $x \in \om \cdot k$ is the $k$th element of $\Commonconst_i$, we represent it as the ordered pair $(i,k)$.

Before defining the coloring in general, we give an example with $\numpoints=5$ and $k=200$. 
%
We define the color of the element
$$e=\{ (3,12), (50,2), (110,12), (110,7777), (117,3) \}$$
as follows:

\begin{itemize}
\item
Order the ordered pairs by their second coordinates.
\item
If some elements have the same second coordinate, break the tie with their first coordinates. At this point in our example we have
$$((50,2), (117,3) (3,12), (110,12), (110,7777)).$$
\item
Remove the second coordinates to get
$$(50,117,3,110,110).$$
This sequence is the element's color.
\end{itemize}


In general, for any $e=\{(i_1,x_1),\ldots,(i_{\numpoints},x_{\numpoints})\}$, order the ordered pairs by their second coordinates, and in case of ties, use their first coordinates to determine the ordering. Then
$\COL(e)$ is the sequence of first coordinates after ordering.

Note that the number of colors is the number of size $\numpoints$ tuples where each 
number is in $\{1,\ldots,k\}$. Hence there are $k^{\numpoints}$ colors.
We leave it to the reader to show that there can be no $(\om \cdot k)$-$(k^{\numpoints}-1)$-homogeneous $H$. The key idea of the proof, much like the previous lower bounds in this paper, is to use the property of order-equivalence to find elements of an arbitrary $H$ that produce each of the $k^{\numpoints}$ colors.

By this proof, $T(\numpoints,\om \cdot k) \geq k^{\numpoints}$, and so by Theorem~\ref{th:omka-lt}, $T(\numpoints,\om \cdot k) = k^{\numpoints}$.

\end{proof}

\section{The Big Ramsey Degree \texorpdfstring{$T(2,\om^2)$}{T(2, omega 2)}}\label{se:omega2} 

This section provides a concrete example involving ordinals greater than $\om$, in particular $\om^2$, the ordered set of all ordinals $\om\cdot a +b$ with $a,b \in \om$. We first define some notation.
\begin{definition} \label{def:coef}
If $e \in \binom{\om^d}{n},$ we refer to $e$ as an {\it edge.} Then $e$ has $n$ {\it elements} selected from $\om^d.$
Hence each element of $e$ is some ordinal number less than $\om^d.$ That element can be represented as a polynomial of degree $< d$ in $\omega.$ We refer to the natural-valued coefficients of these polynomials as the {\it coefficients} of $e.$
\end{definition}

Proving $T(1, \om^2) = 1$ is a straightforward exercise. We begin with $T(2, \om^2)$ before proceeding to the general case.

\begin{theorem}\label{th:om2} 
The big Ramsey degree $T(2,\om^2)=4$.
\end{theorem} 
\begin{proof}
We first prove $T\left(2, \om^2\right) \leq 4$.

Let $\numcolors \in \N$ and let \begin{align*}
    \COL \colon \binom{\om^2}{2} \to [\numcolors]
\end{align*}
be a $\numcolors$-coloring of $\binom{\om^2}{2}$. 
We define four functions $f_0,f_1,f_2,f_3$ from 
domain $\binom{\om}{4}$ to codomain $\binom{\om^2}{2}$ and then use them to define a coloring 
from 
$\binom{\om}{4}$ to $[\numcolors]^4$.
In what follows, 
let $x_1<x_2<x_3<x_4$. For $i \in [4],$ we define $f_i \colon \binom{\om}{4} \to \binom{\om^2}{2}$ where \begin{align*}
f_0(x_1,x_2,x_3,x_4)  &= \{\om \cdot x_1 + x_2, \om \cdot x_3 + x_4\}\\
f_1(x_1,x_2,x_3,x_4)  &= \{\om \cdot x_1 + x_3, \om \cdot x_2 + x_4\}\\
f_2(x_1,x_2,x_3,x_4)  &= \{\om \cdot x_1 + x_4, \om \cdot x_2 + x_3\}\\
f_3(x_1,x_2,x_3,x_4)  &= \{\om \cdot x_1 + x_2, \om \cdot x_1 + x_3\}.
\end{align*}

Then let
$\COL' \colon \binom{\om}{4} \to [\numcolors]^4$ where
$$\COL'(X) =(\COL(f_0(X)),
\COL(f_1(X)),
\COL(f_2(X)),
\COL(f_3(X))
).
$$

Apply Theorem~\ref{th:ramsey} on $\COL'$ to find some $G \oequiv \omega$ where 
$\left|\COL' \left( \binom{G}{4} \right)\right|=1$.
Let $G = \{\commonconst_0 <\nobreak \commonconst_1 <\nobreak \cdots\}.$ Let
\begin{align*}
H = & ~~~~\!~\{\om \cdot \commonconst_{1} + \commonconst_2, \om \cdot \commonconst_{1} + \commonconst_6,~\om \cdot \commonconst_{1} + \commonconst_{10}, \ldots\} \\
 & \setconcat ~\{\om \cdot \commonconst_{3} + \commonconst_{4}, \om \cdot \commonconst_{3} + \commonconst_{12}, \om \cdot \commonconst_{3} + \commonconst_{20}, \ldots\}\\
 & \setconcat~\{\om \cdot \commonconst_{5} + \commonconst_{8}, \om \cdot \commonconst_{5} + \commonconst_{24}, \om \cdot \commonconst_{5} + \commonconst_{40}, \ldots\}\\
& ~~\!\vdots
\end{align*}
Formally, $$
H = \Commonconst_1 \setconcat \Commonconst_2 \setconcat \cdots
$$ where $$
\Commonconst_i = \{\om\cdot \commonconst_{2i-1} + \commonconst_j \colon j = 2^i + k 2^{i+1}, k \in \N\}.
$$
Note that each of the $A_i$'s coefficients
(as defined in Definition~\ref{def:coef}) 
are disjoint.

Then $H \oequiv \om^2$, as it is the concatenation of countably infinite sets order-equivalent to $\om$. Note that for every $\om \cdot \commonconst_i + \commonconst_j \in H$ we have $\commonconst_i < \commonconst_j$.

Let the single 4-tuple that $\COL'$ outputs on $\binom{G}{4}$ be $Y.$ Consider any edge $\{\om \cdot \rareconst_1 + \rareconst_2, \om \cdot \rareconst_3 + \rareconst_4\} \in \binom{H}{2}$ with $\om \cdot \rareconst_1 + \rareconst_2 < \om \cdot \rareconst_3 + \rareconst_4$. \begin{itemize}
\item If $\rareconst_1 \neq \rareconst_3$, then the two elements are from different $A_i$ sets. Then $\rareconst_1 < \rareconst_3$ by the ordering of the two elements and $\rareconst_2 \neq \rareconst_4$ by the construction of $H$. 
We also have $\rareconst_1 < \rareconst_2$, $\rareconst_1 < \rareconst_4$, and $\rareconst_3 < \rareconst_4$ by the construction of $H$.
There are three subcases: \begin{itemize}
    \item $\rareconst_1<\rareconst_2<\rareconst_3<\rareconst_4$: then $\COL(f_0(\rareconst_1,\rareconst_2,\rareconst_3,\rareconst_4)) \in Y.$
    \item $\rareconst_1<\rareconst_3<\rareconst_2<\rareconst_4$: then $\COL(f_1(\rareconst_1,\rareconst_3,\rareconst_2,\rareconst_4)) \in Y.$
    \item $\rareconst_1<\rareconst_3<\rareconst_4<\rareconst_2$: then $\COL(f_2(\rareconst_1,\rareconst_3,\rareconst_4,\rareconst_2)) \in Y.$
\end{itemize}
In all subcases, $\COL(\{\om \cdot \rareconst_1 + \rareconst_2, \om \cdot \rareconst_3 + \rareconst_4\}) \in Y.$
\item If $\rareconst_1 = \rareconst_3$, then the two elements are from the same $A_i.$ Then $\rareconst_2 < \rareconst_4$ by the ordering of the elements and so $\rareconst_1=\rareconst_3<\rareconst_2<\rareconst_4$ by the construction of $H$. Because $\COL(f_3(\rareconst_1,\rareconst_2,\rareconst_4,\rareconst_4+1)) \in Y$ (note that the output of $f_3$ does not depend on its final argument), $\COL(\{\om \cdot \rareconst_1 + \rareconst_2, \om \cdot \rareconst_3 + \rareconst_4\}) \in Y$.
\end{itemize}
In all cases, $\COL(\{\om \cdot \rareconst_1 + \rareconst_2, \om \cdot \rareconst_3 + \rareconst_4\}) \in Y$ so \begin{align*}
\COL\left(\binom{H}{2}\right) \subseteq Y
\end{align*} with $|Y|=4$. Because $H \oequiv \om^2$ and $\COL$ was arbitrary, $T\left(2, \om^2\right) \leq 4$.

We now prove $T\left(2, \om^2\right) \geq  4$.
Let $\COL \colon \binom{\om^2}{2} \to [4]$ with \begin{align*}
\COL(\om \cdot \commonconst_1 + \commonconst_2, \om \cdot \commonconst_3 + \commonconst_4) = 
\begin{cases}
0 & \commonconst_1 < \commonconst_2 < \commonconst_3 < \commonconst_4\\
1 & \commonconst_1 < \commonconst_3 < \commonconst_2 < \commonconst_4\\
2 & \commonconst_1 < \commonconst_3 < \commonconst_4 < \commonconst_2\\
3 & \textnormal{otherwise}
\end{cases}
\end{align*}
where $\om \cdot \commonconst_1 + \commonconst_2 < \om \cdot \commonconst_3 + \commonconst_4$. 
Let $H \subseteq \omega^2$ be any set such that $H\oequiv \omega^2$. 
We show that $\left|\COL\left(\binom{H}{2}\right)\right)|=4$. Every element of $H$ is of the form $\omega \cdot \rareconst + e,$ for finite $\rareconst,e.$ Because $H \oequiv \om^2$, we have $H = H_1+H_2+ \cdots$ where each $H_i \oequiv \omega$. For each $H_i,$ there exists some $\rareconst_i$ such that $$
H'_i = \{\omega \cdot \rareconst_i + e \in H_i\} \oequiv \om.
$$ Let $e_{i,j}$ be such that every element of $H'_i$ is of the form $\omega \cdot \rareconst_i + e_{i,j},$
where
$$\rareconst_1 < \rareconst_2 < \cdots$$
and for all $i$, 
$$e_{i,1} < e_{i,2} < \cdots.$$
We now show each color is output by $\COL$ on $H$.
\begin{itemize}
    \item Color 0: Starting with any $\rareconst_i$, find some $j$ where $e_{i, j} > \rareconst_i$. This is guaranteed, as $\rareconst_i$ is finite and the $e_{i, j}$ are infinitely increasing. Then find some $k$ where $\rareconst_k > e_{i, j}$; again guaranteed because $e_{i, j}$ is finite and the $\rareconst_i$ are infinitely increasing. Finally, find an
    $\ell$ where $e_{k ,\ell} > \rareconst_k$ guaranteed by similar means. Then 
    $$\COL(\omega\cdot \rareconst_i + e_{i, j}, \omega\cdot \rareconst_k + e_{k, \ell})=0.$$
    \item Colors 1 and 2 are guaranteed to exist by arguments similar to the argument for color 0.
    \item Color 3: The case of $\commonconst_1=\commonconst_3<\commonconst_2<\commonconst_4$ falls into the category of {\it otherwise}. With this in mind, we can start with some $\rareconst_i$, and then find a $e_{i, j} > \rareconst_i$. Then, we only need to find a $e_{i, \ell} > e_{i, j}$; this is guaranteed because $e_{i, j}$ is finite. Then $$\COL(\omega\cdot \rareconst_i + e_{i, j}, \omega\cdot \rareconst_i + e_{i, \ell})=3.$$
\end{itemize}

 
\end{proof}

\section{\capcrules}\label{se:strong} 

We introduce a concept called {\it \crulesLong} (hereafter \crulesnospace) to prove 
general results about
\linebreak
$T(\numpoints,\omega^d \cdot\nobreak k)$.
The concept behind \crules is built on the
ideas of Blass et al.~\cite{BDR21}.
We motivate the concept by examining a previous proof.

The proof that $T(2, \om^2) = 4$ (Theorem~\ref{th:om2})
used four functions $f_0,f_1,f_2$, and $f_3$. 
These functions were chosen to cover $H$ in a way where the color of every edge in $H$ 
was output by one of $f_0,f_1,f_2,$ or $f_3$.
Note that our proof did {\it not} use
$f\colon \binom{\om}{4} \to \binom{\om^2}{2}$ with
$$f(x_1,x_2,x_3,x_4) = \{\om \cdot x_1 + x_3, \om \cdot x_2 + x_3\},$$
for $x_1 < x_2 < x_3 < x_4.$

We did not use $f$ in the proof of $T(2, \om^2) = 4$
because $f$ does not cover any edges in $H$: we constructed $H$ in a 
way where distinct copies of $\om$ had distinct finite coefficients. When $x_1 \neq x_2$, no matter the values of $x_1, x_2,$ and 
$x_3$, the elements 
$\om \cdot x_1 + x_3$ and $\om \cdot x_2 + x_3$ cannot both be from $H$. Our goal in lower bound
proofs is to use as few of these $f_i$ functions as possible to cover some carefully constructed
$H.$

We define a notion of colorings that each of $f_0,f_1,f_2,$ and $f_3$ satisfy for $\binom{\om^2}{2}.$
We also show how to count these colorings, and how these colorings are linked to big Ramsey degrees.

\begin{definition}\label{def:coloring-rule}
We now define \crules (\crulesLong) rigorously. We impose a structure on edges and list criteria that \crules must satisfy.
For $\numpoints,d,k \in \N$, an edge $e = \{p_1, \dots, p_{\numpoints}\}$ in $\binom{\om^d \cdot k}{\numpoints}$ consists of $\numpoints$ elements of $\om^d\cdot k$. Each element $p_{\inumpoints} $ is equal to 
$$\om^d \cdot \rareconst_{\inumpoints} + \om^{d-1} \cdot \commonconst_{\inumpoints,d-1} + \om^{d-2} \cdot \commonconst_{\inumpoints,d-2} + \cdots + \om^1 \cdot \commonconst_{\inumpoints,1} + \commonconst_{\inumpoints,0},$$
where $\commonconst_{\inumpoints,\idegr} \geq 0$ and $0\leq \rareconst_{\inumpoints} < k$.

Thus, any edge $e$ is defined by the $\numpoints$ values of the $\rareconst_{\inumpoints}$ and the $\numpoints\cdot d$ values of the $\commonconst_{\inumpoints,\idegr}$. 

A \textit{\crule (\cruleLong)} on $\binom{\om^d \cdot k}{\numpoints}$ is a pair $(\assny,\xleq)$ of constraints on these values $\rareconst_{\inumpoints}$ and $\commonconst_{\inumpoints,\idegr}$ satisfying certain properties we enumerate below. Firstly, $\assny$ is an assignment of the values for the $\rareconst_{\inumpoints}$; formally, it is a map $\assny: [\numpoints] \to \{0,\dots,k-1\}$ from indices of the $\rareconst_{\inumpoints}$ to the values of the $\rareconst_{\inumpoints}$. Having $\assny(\inumpoints) = v$ means we constrain $\rareconst_{\inumpoints} = v$.
Meanwhile, $\xleq$ is a total preorder (i.e. every two of elements are comparable, and $\xleq$ is reflexive and transitive) on the indices of the $\commonconst_{\inumpoints,\idegr}$. Having $(\inumpoints,\idegr) \xleq (\inumpoints',\idegr')$ means we constrain $\commonconst_{\inumpoints,\idegr} \leq \commonconst_{\inumpoints',\idegr'}$.

We often denote $\assny$ using an ordered list of clauses: $$\rareconst_1 = \assny(1), \rareconst_2 = \assny(2), \dots, \rareconst_{\numpoints} = \assny(\numpoints).$$

We often denote $\xleq$ by a permutation of the expressions $\commonconst_{1,0}, \dots, \commonconst_{\numpoints,d-1}$ interspersed by either $<$ or $=$. We write $\commonconst_{\inumpoints,\idegr} < \commonconst_{\inumpoints',\idegr'}$ to mean that both $(\inumpoints,\idegr) 
\xleq (\inumpoints',\idegr')$ and $(\inumpoints',\idegr') 
\not \xleq (\inumpoints,\idegr)$ and we write $\commonconst_{\inumpoints,\idegr} = \commonconst_{\inumpoints',\idegr'}$ to mean both $(\inumpoints,\idegr) 
\xleq (\inumpoints',\idegr')$ and $(\inumpoints',\idegr') 
\xleq (\inumpoints,\idegr)$. Note that $x = y < z$ and $y = x < z$ are two representations of the same preorder, so this notation is not unique. One example of such a representation is $$
\commonconst_{1, 1} = \commonconst_{2, 1} < \commonconst_{1 ,0} < \commonconst_{2 ,0}.
$$

We refer to the expressions $\commonconst_{\inumpoints,\idegr}$ and $\rareconst_\inumpoints$ when used in our notation for $\assny$ and $\xleq$ as variables, and will describe different preorders $\xleq$ by discussing permutations of some sequence of $\commonconst_{\inumpoints,\idegr}$.

To be a \crule the following must hold:

\begin{enumerate}
    \item \label{cr:index-ordering} If $d \geq 1$, then for all $\inumpoints < \inumpoints',$ $\commonconst_{\inumpoints, 0} < \commonconst_{\inumpoints' ,0}.$ Otherwise when $d = 0$, $\rareconst_{\inumpoints} < \rareconst_{\inumpoints'}$ for all $\inumpoints<\inumpoints'$. (The element indices are ordered by their lowest-exponent variable.) 
    \item \label{cr:y-distinct} For all $\inumpoints,\inumpoints'$, $(\exists \idegr\; \commonconst_{\inumpoints,\idegr} = \commonconst_{\inumpoints',\idegr})\implies \rareconst_{\inumpoints} = \rareconst_{\inumpoints'}$. (Two elements with any $\commonconst$ values that are the same must have the same $\rareconst$ value.)
    \item \label{cr:high-first} For all $\idegr > \idegr',$ $\commonconst_{\inumpoints, \idegr} < \commonconst_{\inumpoints, \idegr'}$. (The high-exponent variables of each element are strictly less than the low-exponent variables.)
    \item \label{cr:term-dim-diff} $\commonconst_{\inumpoints, \idegr} = \commonconst_{\inumpoints', \idegr'} \implies \idegr = \idegr'$. (Only variables with the same exponent can be equal.)
    \item \label{cr:high-split} $\commonconst_{\inumpoints, \idegr} \neq \commonconst_{\inumpoints', \idegr} \implies \commonconst_{\inumpoints,\idegr-1} \neq \commonconst_{\inumpoints',\idegr-1}$ for all $\idegr>0$. (Two elements that differ in a high-exponent variable differ in all lower-exponent variables.)
\end{enumerate}

An example of a \crule for $\binom{\om^2}{2}$ is $$\rareconst_1=0,~\rareconst_2=0,~\commonconst_{1, 1} = \commonconst_{2, 1} < \commonconst_{1, 0} < \commonconst_{2, 0}.$$
Note that because $k=1$ in the example, we must have $\rareconst_{\inumpoints} = 0$ for every $\inumpoints$.


\end{definition}

\begin{definition}\label{def:crule-equivalence}
    Two \crules are \textit{equivalent} whenever their constraints $\assny, \xleq$ are
    exactly equal.
\end{definition}

\begin{definition} 
The \textit{size} of a \crule is how many equivalence classes its $\commonconst_{\inumpoints,\idegr}$ form under $\xleq$: 
for example, $\commonconst_{1, 1} = \commonconst_{2, 1} < \commonconst_{1, 0} < \commonconst_{2, 0}$ would have size $p=3$ regardless of $\rareconst_{\inumpoints}$. Therefore a \crulenospace's size $p$ can be no larger than $\numpoints \cdot d$.
\end{definition}
\begin{definition}~
\begin{enumerate}
\item 
$P_p \left(\numpoints,\om^d \cdot k \right)$ is the number of \crules of size $p$ there are for $\binom{\omega^d \cdot k}{\numpoints}$. 
\item 
$P\left(\numpoints,\om^d \cdot k\right)$ is the total number of \crules there are for $\binom{\omega^d \cdot k}{\numpoints}$ of any size. It can be calculated as $$
\sum_{p=0}^{\numpoints \cdot d} P_p(\numpoints, \om^d \cdot k).
$$
\end{enumerate}
\end{definition}
Eventually,
we will show $T(\numpoints,\om^d \cdot k) = P(\numpoints,\om^d \cdot k)$.

\begin{definition}\label{def:crule-coef}
An edge $$\{\om^d \cdot \rareconst_{\inumpoints} + \om^{d-1} \cdot \commonconst_{\inumpoints,d-1} + \om^{d-2} \cdot \commonconst_{\inumpoints,d-2} + \cdots + \om^1 \cdot \commonconst_{\inumpoints,1} + \commonconst_{\inumpoints,0} \colon 1 \leq \inumpoints \leq \numpoints\}$$
\textit{satisfies} some \crule if $\assny(\inumpoints) = \rareconst_{\inumpoints}$ for every $1 \leq \inumpoints \leq \numpoints$ and if 
$(\inumpoints,\idegr) \xleq (\inumpoints',\idegr') \iff \commonconst_{\inumpoints,\idegr} \leq \commonconst_{\inumpoints',\idegr'}$ for all $1 \leq \inumpoints,\inumpoints' \leq \numpoints$ and $0 \leq \idegr, \idegr' < d.$
Note that while every edge describes some map $\assny$ and order $\xleq,$ if those constructs
do not meet the requirements in Definition~\ref{def:coloring-rule},
that edge does not satisfy any \crulenospace. We extend and formalize the idea of
coefficients introduced in Definition~\ref{def:coef} to be the set of
$a_{i,j}$ terms.
\end{definition}

\section{Big Ramsey Degrees of \texorpdfstring{$\om^d$}{omega d}}\label{se:ord} 
This section is devoted to the case where $k=1$ in $\binom{\om^d \cdot k}{\numpoints}$. When $k=1$, each $\rareconst_{\inumpoints}$ in a \crule is forced to be 0. Then all $\rareconst_{\inumpoints}$ values are the same, so criterion~\ref{cr:y-distinct} of Definition~\ref{def:coloring-rule} is always satisfied. In this section, our proofs focus only on how $\xleq$ permutes the $\commonconst_{\inumpoints,\idegr}$ variables. When values for $\rareconst_{\inumpoints}$ are not specified, they are assumed to be all 0 by default.

We show equality between big Ramsey degrees and counts of \crules in this
section. We start with a recurrence that counts \crulesnospace.

\begin{lemma}
\label{le:P-recur}
For $\numpoints,d \in \N$,


\[
P_p\left( \numpoints,\om^d\right) = \begin{cases}
    0 & \textnormal{(1)\phantom{0} } d = 0 \land \numpoints \geq 2\\
    1 & \textnormal{(2)\phantom{0} } \numpoints = 0 \land p = 0\\
    0 & \textnormal{(3)\phantom{0} } \numpoints = 0 \land p \geq 1\\
    1 & \textnormal{(4)\phantom{0} } d = 0 \land \numpoints = 1 \land p = 0\\
    0 & \textnormal{(5)\phantom{0} } d = 0 \land \numpoints = 1 \land p \geq 1\\
    1 & \textnormal{(6)\phantom{0} } d = 1 \land \numpoints \geq 1 \land \numpoints = p\\
    0 & \textnormal{(7)\phantom{0} } d = 1 \land \numpoints \geq 1 \land \numpoints \neq p\\
    0 & \textnormal{(8)\phantom{0} } d \geq 2 \land \numpoints \geq 1 \land p = 0\\
    \sum\limits_{j=1}^\numpoints \sum\limits_{i=0}^{p-1} \binom{p-1}{i} P_i \left( j,\om^{d-1} \right) P_{p-1-i} \left( \numpoints-j,\om^d\right) & \textnormal{(9)\phantom{0} } d \geq 2 \land \numpoints \geq 1 \land p \geq 1
\end{cases}
\]

\end{lemma}

\begin{proof}

\case{1} First, suppose $\numpoints \geq 2$ and $d = 0$. As argued at the beginning of this section, $\rareconst_{\inumpoints}= 0$ for all $\inumpoints$. But by criterion~\ref{cr:index-ordering} of Definition~\ref{def:coloring-rule}, since $d = 0$ we need $\rareconst_1 < \rareconst_2$, so no \crules are possible regardless of size $p$. This proves the first case of the result.

\case{2} Suppose $\numpoints = 0$. 
Since there are no $\rareconst_{\inumpoints}$ or $\commonconst_{\inumpoints,\idegr}$, the criteria are vacuously satisfied. There is only one \crule and it has size $p = 0$. This proves the second and third cases of the result.

\case{3} When both $d = 0$ and $\numpoints \leq 1$, 
criterion~\ref{cr:index-ordering} of Definition~\ref{def:coloring-rule} is
vacuously satisfied with either no $\rareconst_{\inumpoints}$ or $\rareconst_1 = 0$. Again, because $\numpoints \cdot d = 0$, there are no $\commonconst_{\inumpoints,\idegr}$. Thus there is only one \crule of size $p=0$, which proves the fourth and fifth cases of the result.

\case{4} Now suppose $\numpoints \geq 1$ and $d = 1$. To ensure criterion~\ref{cr:index-ordering} of Definition~\ref{def:coloring-rule}, each of the $\numpoints$ variables $\commonconst_{\inumpoints,0}$ can only form one \crule $\commonconst_{1,0} < \commonconst_{2,0} < \cdots < \commonconst_{\numpoints,0}$ of size $\numpoints$ so $P_{\numpoints}\left(\numpoints, \om^d\right) = 1$ and $P_p\left(\numpoints, \om^d\right) = 0$ for $p \neq \numpoints$. This proves the sixth and seventh cases of the result.

\case{5} Finally, consider $\numpoints \geq 1, d \geq 2$. By the definition of a \crulenospace, because $\numpoints \cdot d > 0,$ there are no \crules of size $p=0.$ This proves the eighth
case of the result.
We prove the final case by showing that the process described below creates all possible \crules of an expression.

Let $\numpoints \geq 1$, $d \geq 2$, and $p \geq 1$ be naturals. Let $1 \leq m \leq \numpoints$ and $0 \leq s \leq p-1$ be naturals. As we proceed through the process, we work with an example of $\numpoints = 4$, $d = 5$, $p=13$, $m=2$, and $s=5$.

We create $$\binom{p-1}{s} P_s \left( m,\om^{d-1} \right) P_{p-1-s} \left( \numpoints-m,\om^d \right)$$ \crules of size $p$ by combining $P_s \left( m,\om^{d-1} \right)$ \crules of size $s$ and $P_{p-1-s} \left( \numpoints-m,\om^d \right)$ \crules of size $p-1-s$, with $\binom{p-1}{s}$ \crules derived from each pair.

Let $\tau_1$ represent one of the $P_s\left(m, \om^{d-1}\right)$ \crules of $\binom{\om^{d-1}}{m}$ of size $s$, and $\tau_2$ represent one of the $P_{p-1-s}\left(\numpoints-m, \om^d\right)$ \crules of $\binom{\om^{d}}{\numpoints-m}$ of size $p-1-s$. In our example, let \begin{align*}
\tau_1 & \colon \commonconst_{1,3} = \commonconst_{2,3} < \commonconst_{1,2} = \commonconst_{2,2} < \commonconst_{1,1} = \commonconst_{2,1} < \commonconst_{1,0} < \commonconst_{2,0}\\
\tau_2 & \colon \commonconst_{1,4} = \commonconst_{2,4} < \commonconst_{1,3} = \commonconst_{2,3} < \commonconst_{1,2} = \commonconst_{2,2} < \commonconst_{2,1} < \commonconst_{1,1} < \commonconst_{1,0} < \commonconst_{2,0}.
\end{align*}

Then we can combine each $\tau_1$ and $\tau_2$ to form $\binom{p-1}{s}$ unique new \crules of size $p$.
In summary, we will interleave the rules together, and add a final equivalence class at the start of
the \crulenospace.
Change the index of each variable $\commonconst_{\inumpoints,\idegr}$ in $\tau_2$ to $\commonconst_{\inumpoints+m,\idegr}$, and intertwine the equivalence classes of the \crules together, preserving each \crulenospace's original ordering of its own equivalence classes: there are $\binom{p-1}{s}$ ways to do this. In our example, after changing the indices of $\tau_2$ we have \begin{align*}
\tau_1 & \colon \commonconst_{1,3} = \commonconst_{2,3} < \commonconst_{1,2} = \commonconst_{2,2} < \commonconst_{1,1} = \commonconst_{2,1} < \commonconst_{1,0} < \commonconst_{2,0}\\
\tau_2 & \colon \commonconst_{3,4} = \commonconst_{4,4} < \commonconst_{3,3} = \commonconst_{4,3} < \commonconst_{3,2} = \commonconst_{4,2} < \commonconst_{4,1} < \commonconst_{3,1} < \commonconst_{3,0} < \commonconst_{4,0}
\end{align*} and one of the $\binom{12}{5}$ permutations is \begin{align*}
\commonconst_{3,4} = \commonconst_{4,4} < \commonconst_{1,3} = \commonconst_{2,3} < \commonconst_{1,2} = \commonconst_{2,2} < \commonconst_{3,3} = \commonconst_{4,3} < \commonconst_{3,2} = \commonconst_{4,2} \\ < \commonconst_{1,1} = \commonconst_{2,1} < \commonconst_{4,1} < \commonconst_{1,0} < \commonconst_{3,1} < \commonconst_{3,0} < \commonconst_{2,0} < \commonconst_{4,0}.
\end{align*}This new \crule likely breaks criterion~\ref{cr:index-ordering} of Definition~\ref{def:coloring-rule}; for each $1 \leq \inumpoints \leq \numpoints$, change the $i$ indices of each $\commonconst_{\inumpoints,\idegr}$ according to where their $\commonconst_{\inumpoints,0}$ is in the ordering of all $\commonconst_{i,0}$. In our example, we have $\commonconst_{1,0} < \commonconst_{3,0} < \commonconst_{2,0} < \commonconst_{4,0}$; after swapping indices $2$ and $3$ to enforce criterion~\ref{cr:index-ordering}, we have \begin{align*}
\commonconst_{2,4} = \commonconst_{4,4} < \commonconst_{1,3} = \commonconst_{3,3} < \commonconst_{1,2} = \commonconst_{3,2} < \commonconst_{2,3} = \commonconst_{4,3} < \commonconst_{2,2} = \commonconst_{4,2} \\ < \commonconst_{1,1} = \commonconst_{3,1} < \commonconst_{4,1} < \commonconst_{1,0} < \commonconst_{2,1} < \commonconst_{2,0} < \commonconst_{3,0} < \commonconst_{4,0}.
\end{align*}

There are now $d \cdot (\numpoints - m) + (d - 1) \cdot m = \numpoints \cdot d - m$ variables in the \crulenospace. There are $m$ variables of the form $\commonconst_{\inumpoints_s,d-1}$ for $1 \leq s \leq m$ that are not in the \crule yet; insert one equivalence class $\commonconst_{\inumpoints_1,d-1} = \commonconst_{\inumpoints_2,d-1} = \cdots = \commonconst_{\inumpoints_m,d-1}$ at the front of the new \crulenospace, bringing its size to $p$. We insert $\commonconst_{1,4}=\commonconst_{3,4}$ in our example to get \begin{align*}
\commonconst_{1,4}=\commonconst_{3,4} < \commonconst_{2,4} = \commonconst_{4,4} < \commonconst_{1,3} = \commonconst_{3,3} < \commonconst_{1,2} = \commonconst_{3,2} < \commonconst_{2,3} = \commonconst_{4,3} < \commonconst_{2,2} = \commonconst_{4,2} \\ < \commonconst_{1,1} = \commonconst_{3,1} < \commonconst_{4,1} < \commonconst_{1,0} < \commonconst_{2,1} < \commonconst_{2,0} < \commonconst_{3,0} < \commonconst_{4,0}.
\end{align*} 
Each \crule is unique by the $\tau_1$ and $\tau_2$ used to create it because the process is invertible: we can remove the leading equivalence class and separate the remaining variables into $\tau_1$ and $\tau_2$ by whether their indices were in the leading equivalence class and changing the indices. Here, $\tau_1$ corresponds to indices 1 and 3 (not including the leading equivalence class) and is \textbf{bolded}, and $\tau_2$ corresponds to indices 2 and 4 and is \underline{underlined}.
\begin{align*}
\commonconst_{1,4}=\commonconst_{3,4} < \underline{\commonconst_{2,4} = \commonconst_{4,4}} < \mathbf{\commonconst_{1,3} = \commonconst_{3,3}} < \mathbf{\commonconst_{1,2} = \commonconst_{3,2}} < \underline{\commonconst_{2,3} = \commonconst_{4,3}} < \underline{\commonconst_{2,2} = \commonconst_{4,2}} \\ < \mathbf{\commonconst_{1,1} = \commonconst_{3,1}} < \underline{\commonconst_{4,1}} < \mathbf{\commonconst_{1,0}} < \underline{\commonconst_{2,1}} < \underline{\commonconst_{2,0}} < \mathbf{\commonconst_{3,0}} < \underline{\commonconst_{4,0}}.
\end{align*}

We now show that each \crule created by this process has the properties described by Definition~\ref{def:coloring-rule}. Because the high-exponent equivalence class $\commonconst_{\inumpoints_1,d-1} = \commonconst_{\inumpoints_2,d-1} = \cdots = \commonconst_{\inumpoints_m,d-1}$ was added at the start of the \crulenospace, the high-exponent coefficients of each term are smaller than the low-exponent coefficients (see
Definition~\ref{def:crule-coef} for a definition of coefficients). 
Criterion~\ref{cr:index-ordering} is satisfied by changing the indices of the variables.
The remaining criteria are satisfied because $\tau_1$ and $\tau_2$ satisfied them and their internal orders were preserved in permuting the equivalence classes. Therefore this process does not overcount \crulesnospace.

We next show that every \crule of $\binom{\om^d}{\numpoints}$ is counted by this process. Each can be mapped to some $\tau_1$ and $\tau_2$ that create it by the argument 2 paragraphs above which proves that the process creates unique \crulesnospace. Every \crule of $\binom{\om^d}{\numpoints}$ must have a leading equivalence class of $\commonconst_{\inumpoints_1,d-1} = \commonconst_{\inumpoints_2,d-1} = \cdots = \commonconst_{\inumpoints_m,d-1}$ to 
satisfy Definition~\ref{def:coloring-rule} 
(the equivalence class might only contain one variable); taking only the variables $\commonconst_{\inumpoints,\idegr}$ with indices appearing in that equivalence class (but not those variables in the equivalence class itself) forms $\tau_1$, a \crule for $\binom{\om^{d-1}}{m}$. The variables with $\inumpoints$ indices not in the equivalence class form $\tau_2$, a \crule for $\binom{\om^d}{\numpoints-m}$. The original \crule of $\binom{\om^d}{\numpoints}$ is counted by interleaving $\tau_1$ with $\tau_2$ and inserting the leading equivalence class of $\commonconst_{\inumpoints_1,d-1} = \commonconst_{\inumpoints_2,d-1} = \cdots = \commonconst_{\inumpoints_m,d-1}$.
Therefore the final case of the result holds.
\end{proof}

For more information, see the OEIS~\cite{OEIS}, where $P(\numpoints,\om^2)$ is sequence A000311 and $P(2,\om^d)$ is A079309. Both of these
sequences have combinatorial interpretations outside of Ramsey
theory.
Values for small $\numpoints,d$ are tabulated in the appendix; see Table~\ref{tab:om-d-a}. This table has been added to the OEIS, see sequence A364026. It is worth noting that the recurrence $$
\sum\limits_{j=1}^\numpoints \sum\limits_{i=0}^{p-1} \binom{p-1}{i} P_i \left( j,\om^{d-1} \right) P_{p-1-i} \left( \numpoints-j,\om^d\right)
$$ can be simplified using exponential generating functions of the form $$
F(n,\om^d)[x] = \sum_{p=0}^\infty \frac{1}{p!} P_p(n,\om^d) x^p.
$$

\subsection{\texorpdfstring{$T(\numpoints,\om^d) \leq P(\numpoints,\om^d)$}
 {T(\numpoints, omega d) <= P(\numpoints, omega d)}}

We use the following lemma to show that \crules bound big Ramsey degrees from above.

\begin{lemma}
\label{le:bend-nats}
For $\numpoints, d \in \N$ and some ordered set
$G \subseteq \om$ with $G \oequiv \om$, there exists some $H \subseteq \om^d$ with $H \oequiv \om^d$ where for all $e \in \binom{H}{\numpoints}$, $e$ satisfies a \crule of $\binom{\om^d}{\numpoints}$ and each coefficient of $e$ (see Definition~\ref{def:crule-coef}) is contained in $G$.
\end{lemma}
\begin{proof}
Let $G = \{
\commonconst_1 < \commonconst_2 < \commonconst_3 < \cdots\}$ with each
$\commonconst_{\inumpoints} \in \om$ and $G \oequiv \om$. We proceed by induction on $d$.

When $d=0$, $\om^0 \oequiv 1$ and so $H = \{0\} = 1$ suffices. When $n < 2,$
the only \crule of $\binom{H}{n} = \binom{1}{n}$ is satisfied by every edge
of $\binom{1}{n}$ and has $n \cdot d = 0$ coefficients:
the $1$ in $\om^0 \cdot 1$ is a $b_i$
term, not an $a_{i,j}.$
When $n\geq 2,$ there are no \crules of $\binom{1}{n}$ so the result
is vacuously satisfied.

When $d=1$, $\om^1 \oequiv \om$ so $H=G$ suffices, as the single \crule for $\binom{\om}{n}$
is satisfied by all $n$-subsets of $\om,$ and their $n$ coefficients are in $H,$ hence in $G.$

For $d \geq 2$, partition $G 
$ into infinite sets, each order-equivalent to $\om$: \begin{align*}
A_0 & = \{\commonconst_1,\commonconst_3,\commonconst_5,\ldots\}\\
A_1 & = \{\commonconst_2,\commonconst_6,\commonconst_{10},\ldots\}\\
A_2 & = \{\commonconst_4,\commonconst_{12},\commonconst_{20},\ldots\}\\
A_3 & = \{\commonconst_8,\commonconst_{24},\commonconst_{40},\ldots\}\\
& ~~\vdots
\end{align*}
Formally, $$
\Commonconst_i = \{\commonconst_j \colon j = 2^i + k2^{i+1}, k \in \N \}.
$$
Reserve $\Commonconst_0$ for later. Apply the inductive hypothesis, taking $G$ to be $\Commonconst_i,$ for all $i \geq 1$. This yields $S_i \oequiv \om^{d-1}$ for all $i \geq 1$. Then for all $i \geq 1$, for all $e \in \binom{S_i}{\numpoints}$, $e$ satisfies a \crule of $\binom{\om^{d-1}}{\numpoints}$ and each coefficient of $e$ is contained in $\Commonconst_i$. We now pair each element of $\Commonconst_0$ with an $S_i.$ Let \begin{align*}
H = & ~~~~\!\{\om^{d-1}\commonconst_{1} + \beta : \beta \in S_0\} \\
 & \cup \{\om^{d-1}\commonconst_{3} + \beta : \beta \in S_1\} \\
 & \cup \{\om^{d-1}\commonconst_{5} + \beta : \beta \in S_2 \}\\
& \cup ~\cdots\\
= & \bigcup_{i = 0}^\omega H_i\text{, where } H_i=
\{\om^{d-1}\commonconst_{2i+1} + \beta : \beta \in S_i\},
\end{align*}

ordered by ordinal comparison. Every element of $H_i$ is less than every element of $H_{i+1}$. Also, each $H_i$ is order-isomorphic to $S_i$ and thus $\omega^{d-1}$.
Therefore, $H \oequiv \om^{d}$.

For any edge $e$ in $\binom{H}{\numpoints}$, index its variables to satisfy criterion~\ref{cr:index-ordering} of Definition~\ref{def:coloring-rule} (this is possible because all low-exponent coefficients are distinct in $H$).
Because we constructed the $\om$ rows of $H$ out of $S_i$ with distinct coefficents,
criterion~\ref{cr:y-distinct} is satisfied.
Criterion~\ref{cr:high-first} 
is satisfied inductively for variables with exponents lower than $d-1$. Because $\min \Commonconst_i = \commonconst_{2^i}$ for all $i$ and $2i-1 < 2^i$ for all naturals
$i \geq 1$, $\commonconst_{2i-1}<\commonconst$ for all $\commonconst \in \Commonconst_i$. So criterion~\ref{cr:high-first} is satisfied by $e$. Because $\Commonconst_0$ is disjoint from all $\Commonconst_i$ with $i \geq 1$, criterion~\ref{cr:term-dim-diff} is satisfied for variables with exponent $d-1,$ and by induction, it is satisfied for lower exponents. Because $\Commonconst_i$ is disjoint with $\Commonconst_j$ for all $i \neq j$, elements that differ in variables with exponent $d-1$ differ in all lower-exponent variables. The induction with the previous statement satisfies criterion~\ref{cr:high-split}. Therefore $e$ satisfies a \crule of $\binom{\om^d}{\numpoints}$. The coefficients in $e$ are contained in $G$ by the construction of $H$ from $G$.
\end{proof}

\begin{theorem}\label{th:wd-LT}
For $\numpoints,d \in \N$, $T\left(\numpoints, \om^d\right) \leq P\left(\numpoints, \om^d\right)$.
\end{theorem} 

\begin{proof}
Let $\numcolors \in \N$ and let \begin{align*}
    \COL \colon \binom{\om^d}{\numpoints} \to [\numcolors]
\end{align*}
be a $\numcolors$-coloring of $\binom{\om^d}{\numpoints}$. 

Enumerate the \crules of $\binom{\om^d}{\numpoints}$ as $\tau_0$ to $\tau_{P(\numpoints, \om^d)-1}$. The maximum size of any \crule of $\binom{\om^d}{\numpoints}$ is $\numpoints \cdot d$. For each $\tau_i$, let \begin{align*}
    f_i \colon \binom{\om}{\numpoints \cdot d} \to \binom{\om^d}{\numpoints}
\end{align*}
where if $\tau_i$ has size $p$, $f_i$ maps $X$ to the unique $e \in \binom{\om^d}{\numpoints}$ where $e$ satisfies $\tau_i$ and the $p$ equivalence classes of $e$ are made up of the $p$ least elements of $X$. For example, one \crule of $\binom{\om^2}{2}$ is $$
\commonconst_{1, 1} = \commonconst_{2, 1} < \commonconst_{1, 0} < \commonconst_{2, 0}.
$$ The corresponding $f_i$ would be $f_i \colon \binom{\om}{4} \to \binom{\om^2}{2}$ with \begin{align*}
f_i(x_1,x_2,x_3,x_4) = \{\om \cdot x_1 + x_2, \om \cdot x_1 + x_3\}
\end{align*} where $x_1 < x_2 < x_3 < x_4$. Note that $f_i$ does not depend on $x_4$ -- this is because the example \crule has size 3, but the largest \crule for $\binom{\om^2}{2}$ has size 4.

Then, define $\COL' \colon \binom{\om}{\numpoints \cdot d} \to [\numcolors]^{P(\numpoints,\om^d)}$ with \begin{align*}
    \COL'(X) = (\COL(f_0(X)),\COL(f_1(X)),\ldots,\COL(f_{P(\numpoints, \om^d)-1}(X)))
\end{align*}
and apply Theorem~\ref{th:ramsey} to find some $G \oequiv \om$ where
\begin{align*}
 \left |  \COL' \left( \binom{G}{\numpoints \cdot d} \right)\right | = 1.
\end{align*}
Let the one color in $\COL'\left(\binom{G}{\numpoints \cdot d}\right)$ be $\Rareconst$.
Note that $\Rareconst$ is a tuple of $P(\numpoints, \om^d)$ colors.

Apply Lemma~\ref{le:bend-nats} to find some $H \oequiv \om^d$ with the properties listed in Lemma~\ref{le:bend-nats}. Now we claim \begin{align*}
    \left | \COL \left( \binom{H}{\numpoints} \right)\right |\leq P(\numpoints, \om^d).
\end{align*}

By Lemma~\ref{le:bend-nats}, each element $e \in \binom{H}{\numpoints}$ satisfies a \crule of $\binom{\om^d}{\numpoints}$. 
Then for any arbitrary edge $e$, let $e$ satisfy $\tau_i$ of size $p \leq \numpoints \cdot d$. 
Then take the $p$ unique values in $e$, and if necessary, insert any new larger naturals from 
$G$ to form a set of $\numpoints \cdot d$ values; denote this $\Commonconst \in \binom{G}{\numpoints \cdot d}$.
Then
$\COL'(\Commonconst) = \Rareconst$ so by the definition of $\COL'$, $\COL(e) \in \Rareconst$. 
Because $|\Rareconst| = P(\numpoints,\om^d)$, $T(\numpoints,\om^d) \leq P(\numpoints,\om^d)$.
\end{proof}

\subsection{\texorpdfstring{$T(\numpoints,\omega^d) \geq P(\numpoints, \om^d)$}{T(\numpoints, omega d) >= P(\numpoints, omega d)}}

\begin{theorem}\label{th:wd-GT}
For $\numpoints,d \in \N$, $T\left(\numpoints,\om^d\right) \geq P\left(\numpoints,\om^d\right)$. Therefore by Theorem~\ref{th:wd-LT}, $T(\numpoints,\om^d) =  P(\numpoints,\om^d)$.
\end{theorem} 

\begin{proof}
If $P(\numpoints,\om^d) = 0$, then this is satisfied vacuously because $T(\numpoints, \om^d) \geq 0$. Now suppose $P(\numpoints,\om^d) \geq 1$. Note that all \crules of $\binom{\om^d}{\numpoints}$ are disjoint from each other. That is, for any edge $e \in \binom{\om^d}{\numpoints}$, if $e$ satisfies $\tau'$, then it does not satisfy any nonequivalent \crule of $\binom{\om^d}{\numpoints},$ where
equivalence between \crules is defined by their $\assny$ maps and $\xleq$ orders being
the same. This is because if $e$ were to satisfy two \crules $\tau_1$ and $\tau_2$, then $\tau_1$ and $\tau_2$ must share the same equivalence classes and order, so the \crules must be equivalent. Therefore, we can index them $\tau_0, \ldots, \tau_{P(\numpoints, \om^d)-1}$ and construct a coloring $\COL \colon \binom{\om^d}{\numpoints} \to [P(\numpoints, \om^d)]$ with\begin{align*}
    \COL(e) = \begin{cases}
    i & e \textnormal{ satisfies } \tau_i\\
    0 & \textnormal{otherwise}
    \end{cases}
\end{align*}
Similar to the proof of Theorem~\ref{th:om2}, our coloring has two ways to output color 0, both 
through the satisfaction of $\tau_0$ and through the catch-all case. The part that forces 
color 0 to be present in all order-equivalent subsets is the satisfaction of $\tau_0$.

We now prove there is no $\omega^d$-$(P(\numpoints,\om^d)-1)$-homogeneous set. 
For all $H \oequiv \om^{d}$ and 
for every \crule $\tau$ of $\binom{\om^d}{\numpoints}$, 
we find some $e \in \binom{H}{\numpoints}$ that satisfies $\tau$.
For arbitrary $H \oequiv \om^{d}$ and $\tau$, we find $e_{\inumpoints,\idegr}$ where 
\begin{align*}
    \{\om^{d-1} e_{1,d-1} + \cdots + \om^{1} e_{1,1} + e_{1,0}, \ldots,
    \om^{d-1} e_{\numpoints,d-1} + \cdots + \om^{1} e_{\numpoints,1} + e_{\numpoints,0}\}
\end{align*}
satisfies $\tau$.

We do this by assigning values to each $e_{\inumpoints,\idegr}$ according to where the equivalence class
that contains $\commonconst_{\inumpoints,\idegr}$ is found in $\tau$, moving left to right in the permutation of $\tau$. By criterion~\ref{cr:high-first}
of Definition~\ref{def:coloring-rule}, 
each $e_{\inumpoints,\idegr}$ is assigned before $e_{\inumpoints,\idegr-1}$. As we do this, we ensure that if the leftmost unassigned value in $\tau$ is $e_{\inumpoints,\idegr}$, then \begin{multline*}\{ \om^{d-1} e_{\inumpoints,d-1} + \cdots + \om^{\idegr+1} e_{\inumpoints,\idegr+1} + \om^\idegr c_{\idegr} + \om^{\idegr-1} c_{\idegr-1} + \cdots + \om^1 c_1 + c_0 : c_i \in \om\} \cap H \oequiv \om^{\idegr+1}.\end{multline*}

By criterion~\ref{cr:high-first} of 
Definition~\ref{def:coloring-rule}, 
the leftmost variable in $\tau$ must be $\commonconst_{\inumpoints,d-1}$. Before any values are assigned, it is clear that \begin{align*}\{ \om^{d-1} c_{d-1} + \cdots + \om^1 c_1 + c_0 : c_i \in \om\} = \om^{d},\end{align*} and because $H \subseteq \om^{d}$, $\om^{d} \cap H = H \oequiv \om^{d}$.

By criterion~\ref{cr:term-dim-diff} of 
Definition~\ref{def:coloring-rule}, all variables in an equivalence class must have the same exponent $d$. Let the leftmost unassigned equivalence class in $\tau$ be $\commonconst_{\inumpoints_1,\idegr} = \commonconst_{\inumpoints_2,\idegr} = \cdots = \commonconst_{\inumpoints_m,\idegr}$.
By criterion~\ref{cr:high-first}, each $\commonconst_{\inumpoints_i,\ell}$ for $1 \leq i \leq m$ and $\ell > \idegr$ appeared to the left of this equivalence class and has already been assigned a value, and by criterion~\ref{cr:high-split} the values for each exponent are equal: for all $\ell > \idegr$ and $1 \leq i \leq m$, $e_{\inumpoints_i,\ell} = e_{\inumpoints_1,\ell}$.

By our previous steps, \begin{multline*}\{ \om^{d-1} e_{\inumpoints_1,d-1} + \cdots + \om^{\idegr+1} e_{\inumpoints_1,\idegr+1} + \om^{\idegr} c_{\idegr} + \om^{\idegr-1} c_{\idegr-1} + \cdots + \om^1 c_1 + c_0 : c_i \in \om\} \cap H \oequiv \om^{\idegr+1}.\end{multline*}
Then there exists some value $e'$ where \begin{multline*}\{ \om^{d-1} e_{\inumpoints_1,d-1} + \cdots + \om^{\idegr+1} e_{\inumpoints_1,\idegr+1} + \om^{\idegr} e' + \om^{\idegr-1} c_{\idegr-1} + \cdots + \om^1 c_1 + c_0 : c_i \in \om\} \cap H \oequiv \om^{\idegr},
\end{multline*} where $e'$ is greater than all previously assigned (and therefore finite) $e_{\inumpoints,\idegr}$ values. Then for $1 \leq i \leq m$, assign $e_{\inumpoints_i,\idegr}$ to be $e'$.

We can repeat this process to find $e_{\inumpoints,\idegr}$ for each \crule of $\binom{\om^d}{\numpoints}$ for arbitrary $H \oequiv \om^{d}$. Therefore for all 
$H \oequiv \omega^d$,
$\left|\COL\left(\binom{H}{\numpoints}\right)\right|\geq P(\numpoints,\om^d)$ so
$T(\numpoints,\om^d) \geq P\left(\numpoints, \om^d\right)$. With
Theorem~\ref{th:wd-LT}, we have
$T(\numpoints,\om^d) = P\left(\numpoints, \om^d\right)$.

\end{proof}

\section{Big Ramsey Degrees of \texorpdfstring{$\om^d \cdot k$}{omega d * k}}\label{se:ord2} 

We now use the theory we developed for the case $k=1$ to prove results for arbitrary $k$. We first extend the recurrence from Lemma~\ref{le:P-recur}.

\begin{lemma}
\label{le:P-recur-k}
For $\numpoints,d,k \in \N$,

\begin{multline*}
    P_p\left( \numpoints,\om^d \cdot k\right) = \\
    \begin{cases}
    0 & \textnormal{(1)\phantom{0} } d = 0 \land \numpoints > k\\
    1 & \textnormal{(2)\phantom{0} } \numpoints = 0 \land p = 0\\
    0 & \textnormal{(3)\phantom{0} } \numpoints = 0 \land p \geq 1\\
    \binom{k}{\numpoints} & \textnormal{(4)\phantom{0} } d = 0 \land 1 \leq \numpoints \leq k \land p = 0\\
    0 & \textnormal{(5)\phantom{0} } d = 0 \land 1 \leq \numpoints \leq k \land p \geq 1\\
    k^\numpoints & \textnormal{(6)\phantom{0} } d = 1 \land \numpoints \geq 1 \land \numpoints = p\\
    0 & \textnormal{(7)\phantom{0} } d = 1 \land \numpoints \geq 1 \land \numpoints \neq p\\
    0 & \textnormal{(8)\phantom{0} } d \geq 2 \land \numpoints \geq 1 \land p = 0\\
    k \sum\limits_{j=1}^\numpoints
    \sum\limits_{i=0}^{p-1} \binom{p-1}{i} P_i \left( j,\om^{d-1} \right) P_{p-1-i} \left( \numpoints-j,\om^d \cdot k\right)
    & \textnormal{(9)\phantom{0} } d \geq 2 \land \numpoints \geq 1 \land p \geq 1
    \end{cases}
\end{multline*}

\end{lemma}

\begin{proof}
\case{1} First, suppose $\numpoints > k$ and $d = 0$. Since $0 \leq \rareconst_{\inumpoints} < k$ for all $\rareconst_{\inumpoints}$, there are at most $k$
unique values for the $\rareconst_{\inumpoints}$. But by criterion~\ref{cr:index-ordering} of Definition~\ref{def:coloring-rule}, since $d = 0,$ we need $\numpoints$ unique values of $\rareconst_{\inumpoints}$, so no \crules are possible regardless of size $p$. This proves the first case of the result.

\case{2} Suppose $\numpoints = 0$. Then there can be no $\rareconst_{\inumpoints}$, and since $\numpoints \cdot d = 0$, there can be no $\commonconst_{\inumpoints,\idegr}$. Thus all criteria are vacuously satisfied. Because there are no $\rareconst_{\inumpoints}$ or $\commonconst_{\inumpoints,\idegr}$, there is only one \crulenospace, and it has size $p = 0$. This proves the second and third cases of the result.

\case{3} When both $d = 0$ and $\numpoints \leq k$,
criterion~\ref{cr:index-ordering} of Definition~\ref{def:coloring-rule} can be satisfied with the assignments to the $\numpoints$ variables $\rareconst_{\inumpoints}$ being any permutation of $\numpoints$ unique values out of $k$ possible natural values.
This leads to $\binom{k}{\numpoints}$ feasible combinations. Again, because $\numpoints \cdot d = 0$, there are no variables $\commonconst_{\inumpoints,\idegr}$ to permute so there are $\binom{k}{\numpoints}$ empty \crules of size $p=0$, which proves the fourth and fifth cases of the result.

\case{4} Now suppose $\numpoints \geq 1$ and $d = 1$. To ensure criteria~\ref{cr:index-ordering} of Definition~\ref{def:coloring-rule}, each of the $\numpoints$ variables  $\commonconst_{\inumpoints,0}$ can only form one permutation $\commonconst_{1,0} < \commonconst_{2,0} < \ldots < \commonconst_{\numpoints,0}$ of size $\numpoints$. Because all $\commonconst_{\inumpoints,\idegr}$ are distinct and $d = 1$, the values $\rareconst_{\inumpoints}$ are not restricted by any criteria so each of the $\numpoints$ variables can be any of the $k$ naturals. Therefore $P_{\numpoints}\left(\numpoints, \om^d \cdot k \right) = k^{\numpoints}$ and $P_p\left(\numpoints, \om^d \cdot k \right) = 0$ for $p \neq \numpoints$. This proves the sixth and seventh cases of the result.

\case{5} Finally, consider $\numpoints \geq 1, d \geq 2$. By the definition of a \crule, because $\numpoints \cdot d > 0,$ there are no \crules of size $p=0.$ This proves the eighth
case of our result. We prove the final case by showing the process for combining \crules described below creates all possible \crules of an expression.

Let $\numpoints \geq 1$, $d \geq 2$, $k \geq 0$ be naturals. Let $p \geq 1$, $1 \leq m \leq \numpoints$ and $0 \leq s \leq p-1$ be naturals.
We create $$k \binom{p-1}{s} P_s \left( m,\om^{d-1} \right) P_{p-1-s} \left( \numpoints-m,\om^d \cdot k\right)$$ \crules of size $p$ by combining $P_s \left( m,\om^{d-1} \right)$ \crules of size $s$ and $P_{p-1-s} \left( \numpoints-m,\om^d \cdot k\right)$ \crules of size $p-1-s$, with $k \binom{p-1}{s}$ new \crules for each pair of smaller \crulesnospace.

Let $\tau_1$ represent one of the $P_s\left(m, \om^{d-1}\right)$ \crules of $\binom{\om^{d-1}}{m}$ of size $s$, and $\tau_2$ represent one of the $P_{p-1-s}\left(\numpoints-m, \om^d \cdot k\right)$ \crules of $\binom{\om^{d} \cdot k}{\numpoints-m}$ of size $p-1-s$. 

Then we can combine each $\tau_1$ and $\tau_2$ to form $k \binom{p-1}{s}$ unique new \crules of size $p.$ First, we intertwine them as in the proof of Lemma~\ref{le:P-recur}: change the indices of the $\commonconst_{\inumpoints,\idegr}$ in $\tau_2$ to $\commonconst_{\inumpoints+m,\idegr}$ and combine the equivalence classes of $\tau_1$ and $\tau_2$ while preserving their original orders. Insert a leading equivalence class $\commonconst_{\inumpoints_1,d-1} = \commonconst_{\inumpoints_2,d-1} = \cdots = \commonconst_{\inumpoints_m,d-1}.$ This leads to $\binom{p-1}{s}$ new permutations of the $\commonconst_{\inumpoints,\idegr}$.

Because $\tau_1$ was a \crule for $\binom{\om^{d-1}}{m}$, each of its $\rareconst_{\inumpoints}$ had a value of 0. Now that we are creating a 
\crule for $\binom{\om^d \cdot k}{\numpoints}$, we can choose the $\rareconst_{\inumpoints}$ to be
composed of values between 0 and $k-1$. By criterion~\ref{cr:y-distinct} of Definition~\ref{def:coloring-rule}, because all elements from $\tau_1$ are bound together in a leading high-exponent equivalence class, the $\rareconst_{\inumpoints}$ must all be equal to one value. This leads to $k$ options for that value; with the options of permuting the $\commonconst_{\inumpoints,\idegr}$, this makes $k \binom{p-s}{s}$ ways to create a new \crulenospace.

For the new \crulenospace's values for $\rareconst_{\inumpoints}$, we assign each element originally from $\tau_2$ with its original $\rareconst$ value (likely at a different index due to changing the indices earlier). Then, the remaining elements from $\tau_1$ are given all the same 
$\rareconst$
value from one of the $k$ options.

Each \crule is unique by the $\tau_1$ and $\tau_2$ used to create it because the process is invertible: we can remove the leading equivalence class and separate the remaining variables into $\tau_1$ and $\tau_2$ by whether their indices were in the leading equivalence class and changing the indices. The $\rareconst$ values for $\tau_2$ can be found from the \crulenospace's $\rareconst$ values after reversing the index change, and the $\rareconst$ values for $\tau_1$ are all 0.

We claim each \crule created by this process has the properties described by Definition~\ref{def:coloring-rule}: Because the high-exponent equivalence class $\commonconst_{\inumpoints_1,d-1} = \commonconst_{\inumpoints_2,d-1} = \cdots = \commonconst_{\inumpoints_m,d-1}$ was added at the start of the \crulenospace, the high-exponent coefficients of each term are smaller than the low-exponent coefficients, where coefficients
were defined in Definition~\ref{def:coef}. 
Criterion~\ref{cr:index-ordering} is satisfied by changing the indices of the variables. Criterion~\ref{cr:y-distinct} is met by assigning all elements from $\tau_1$ the same $\rareconst$ value.
The remaining criteria are satisfied because $\tau_1$ and $\tau_2$ satisfied them and their internal orders were preserved in permuting the equivalence classes. Therefore this process does not overcount \crulesnospace.

We also claim that every \crule of $\binom{\om^d\cdot k}{\numpoints}$ is counted by this process: each can be mapped to some $\tau_1$ and $\tau_2$ that create it by a similar argument to proving that the process creates unique \crules 2 paragraphs above. Every \crule of $\binom{\om^d\cdot k}{\numpoints}$ must have a leading equivalence class of $\commonconst_{\inumpoints_1,d-1} = \commonconst_{\inumpoints_2,d-1} = \cdots = \commonconst_{\inumpoints_m,d-1}$ to satisfy Definition~\ref{def:coloring-rule} 
(the equivalence class might only contain one variable); taking only the variables $\commonconst_{\inumpoints,\idegr}$ with indices appearing in that equivalence class (but not those variables in the equivalence class itself) with all-zero $\rareconst$ values forms $\tau_1$, a \crule for $\binom{\om^{d-1}}{m}$. The variables with $\inumpoints$ indices not in the equivalence class with their $\rareconst$ values form $\tau_2$, a \crule for $\binom{\om^d\cdot k}{\numpoints-m}$. The original \crule of $\binom{\om^d \cdot k}{\numpoints}$ was counted by interleaving $\tau_1$ with $\tau_2$ and inserting the leading equivalence class of $\commonconst_{\inumpoints_1,d-1} = \commonconst_{\inumpoints_2,d-1} = \cdots = \commonconst_{\inumpoints_m,d-1}$.
Therefore the final case of the result holds.
\end{proof}

\subsection{\texorpdfstring{$T(\numpoints,\om^d \cdot k) \leq P(\numpoints,\om^d \cdot k)$}{T(\numpoints, omega d k) <= P(\numpoints, omega d k)}}

\begin{lemma}
\label{le:bend-nats-k}
For $\numpoints,d,k \in \N$ and some ordered subset $G \subseteq \om$ with
$G \oequiv \om$, there exists some $H$ with $H \subseteq \om^d \cdot k$, $H \oequiv \om^d \cdot k$ where for all $e \in \binom{H}{\numpoints}$, $e$ satisfies a \crule of $\binom{\om^d \cdot k}{\numpoints}$ and each coefficient of $e$ (see Definition~\ref{def:crule-coef})
is contained in $G$.
\end{lemma}

\begin{proof}
Let $G = \{\commonconst_0 < \commonconst_1 < \commonconst_2 < \cdots \}$.
When $k = 0,$ $H = \{\} = \emptyset$ suffices.

When $k > 0,$ partition $G$ into $k$ distinct copies: \begin{align*}
    G_0 &= \{a_i : i \equiv 0 \mod k\}\\
    G_1 &= \{a_i : i \equiv 1 \mod k\}\\
    &\vdots\\
    G_{k-1} &= \{a_i : i \equiv k-1 \mod k\}.
\end{align*}
Apply Lemma~\ref{le:bend-nats} on each $G_j$ to yield $H_j \oequiv \om^d,$
and let $$H = 
\{\om^{d} \cdot j + h : h \in H_j\}.$$

We claim that all $e \in \binom{H}{n}$ satisfy a \crulenospace. Since we separated $G$ into disjoint orders $G_{\idegr}$, each $H_{\idegr}$ is disjoint from the others.
With each $e' \in \binom{H_j}{n}$
satisfying a \crulenospace,
criteria~\ref{cr:d-index-ordering}, \ref{cr:d-high-first},
\ref{cr:d-term-dim-diff}, and \ref{cr:d-high-split}
of Definition~\ref{def:coloring-rule}
are automatically satisfied: if some
criterion was not satisfiable and the $H_{\idegr}$ remained distinct, that would
mean problematic elements came from the same part, contradicting how each $e' \in \binom{H_j}{n}$
satisfies a \crulenospace.
Criterion~\ref{cr:d-y-distinct} is satisfied by the construction
of $H$ and each of the
$H_{\idegr}$ being distinct from each other.
The coefficients of all edges $e$ are contained in $G$ by the construction of $H$ from $H_j,$
which have coefficients in $G$ by Lemma~\ref{le:bend-nats}.


\end{proof}

\begin{theorem}\label{th:wdk-LT}
For $\numpoints,d,k \in \N$, $T\left(\numpoints, \om^d \cdot k\right) \leq P\left(\numpoints, \om^d \cdot k\right)$.
\end{theorem} 

\begin{proof}
Let $\numcolors \in \N$ and let \begin{align*}
    \COL \colon \binom{\om^d\cdot k}{\numpoints} \to [\numcolors]
\end{align*}
be a $\numcolors$-coloring of $\binom{\om^d\cdot k}{\numpoints}$.

Enumerate the \crules of $\binom{\om^d\cdot k}{\numpoints}$ from $\tau_0$ to $\tau_{P(\numpoints, \om^d \cdot k)}-1$. The maximum size of any \crule of $\binom{\om^d\cdot k}{\numpoints}$ is $\numpoints \cdot d$. For each $\tau_i,$ let its size be $p_i$. Define \begin{align*}
    f_i \colon \binom{\om}{\numpoints \cdot d} \to \binom{\om^d\cdot k}{\numpoints}
\end{align*}
where $f_i$ maps $X$ to the unique $e \in \binom{\om^d\cdot k}{\numpoints}$
where $e$ satisfies $\tau_i$ and the $p_i$ equivalence classes of $e$ are made up of the $p_i$ least elements of $X$. For example, one \crule of $\binom{\om^2\cdot 2}{2}$ is $$
\rareconst_1=0,~\rareconst_2=1,~\commonconst_{1, 1} < \commonconst_{2, 1} < \commonconst_{2, 0} < \commonconst_{1, 0}.
$$ The corresponding $f_i$ would be $f_i \colon \binom{\om}{4} \to \binom{\om^2\cdot 2}{2}$ with \begin{align*}
f_i(x_1,x_2,x_3,x_4) = \{\om^2 \cdot 0 + \om \cdot x_1 + x_4, \om^2 \cdot 1 + \om \cdot x_2 + x_3\}
\end{align*} where $x_1 < x_2 < x_3 < x_4$. Note that the values of the $\rareconst_{\inumpoints}$ are used directly in the definition of $f_i$ -- for the \crule with identical orderings on $\commonconst_{\inumpoints,\idegr}$ except $\rareconst_0=1$ and $\rareconst_1=0$, the coefficients $\rareconst_{\inumpoints}$ would be swapped.

Then, define $\COL' \colon \binom{\om}{\numpoints \cdot d} \to [\numcolors]^{P(\numpoints,\om^d \cdot k)}$ with \begin{align*}
    \COL'(X) = (\COL(f_0(X)),\COL(f_1(X)),\ldots,\COL(f_{P(\numpoints, \om^d \cdot k)-1}(X)))
\end{align*}
and apply Theorem~\ref{th:ramsey} to find some $G \oequiv \om$ where \begin{align*}
\left |    \COL' \left( \binom{G}{\numpoints \cdot d} \right)\right |=1.
\end{align*}
Let $\Rareconst$ be the one color in $\COL'\left(\binom{G}{\numpoints\cdot d}\right)$.
Note that $\Rareconst$ is a tuple of $P(\numpoints, \om^d \cdot k)$ colors.

Apply Lemma~\ref{le:bend-nats-k} to find some $H \oequiv \om^d \cdot k$ with the properties listed in Lemma~\ref{le:bend-nats-k}. Now we claim \begin{align*}
 \left |   \COL \left( \binom{H}{\numpoints} \right)\right | \leq 
 P(\numpoints, \om^d \cdot k).
\end{align*}

By Lemma~\ref{le:bend-nats-k}, each element $e \in \binom{H}{\numpoints}$ satisfies a \crule of $\binom{\om^d\cdot k}{\numpoints}$. 
Then for any arbitrary edge $e$, let $e$ satisfy $\tau_i$ of size $p_i \leq \numpoints \cdot d$. 
Then take the $p_i$ unique values in $e$, and if necessary, insert any new larger nonnegative naturals from 
$G$ to form a set of $\numpoints \cdot d$ values; denote this by $\Commonconst \in \binom{G}{\numpoints \cdot d}$.
Since $\COL'(\Commonconst) = \Rareconst,$ by the definition of $\COL'$, $\COL(e) \in \Rareconst$. 
Because $|\Rareconst| = P(\numpoints,\om^d \cdot k)$, $T(\numpoints,\om^d \cdot k) \leq P(\numpoints,\om^d \cdot k)$.
\end{proof}

\subsection{\texorpdfstring{$T(\numpoints,\omega^d \cdot k) \geq P(\numpoints, \om^d \cdot k)$}{T(\numpoints, omega d k) >= P(\numpoints, omega d k)}}

\begin{theorem}\label{th:wdk-GT}
For $\numpoints,d,k \in \N$, $T\left(\numpoints,\om^d \cdot k\right) \geq P\left(\numpoints,\om^d \cdot k\right)$.
Therefore, by Theorem~\ref{th:wdk-LT}, \linebreak $T(\numpoints,\om^d \cdot k) = P(\numpoints,\om^d \cdot k)$.
\end{theorem}
\begin{proof}
If $P(\numpoints,\om^d \cdot k) = 0$, then this is satisfied vacuously because $T(\numpoints, \om^d \cdot k) \geq 0$. Suppose $P(\numpoints,\om^d \cdot k) \geq 1$. Note that all \crules of $\binom{\om^d \cdot k}{\numpoints}$ are disjoint from each other. That is, for any edge $e \in \binom{\om^d \cdot k}{\numpoints}$, if $e$ satisfies $\tau'$, then it does not satisfy any nonequivalent \crule of $\binom{\om^d \cdot k}{\numpoints}$. This is because if $e$ were to satisfy two \crules $\tau_1$ and $\tau_2$, then $\tau_1$ and $\tau_2$ must share the same $\rareconst_{\inumpoints}$ values, equivalence classes, and order, so the \crules must be equivalent. Therefore, we can index them $\tau_0, \tau_1, \ldots, \tau_{P(\numpoints, \om^d)-1}$ and construct a coloring $\COL \colon \binom{\om^d \cdot k}{\numpoints} \to [P(\numpoints, \om^d\cdot k)]$ with \begin{align*}
    \COL(e) = \begin{cases}
    i & e \textnormal{ satisfies } \tau_i\\
    0 & \textnormal{otherwise}.
    \end{cases}
\end{align*}
Similar to Theorem~\ref{th:om2}, our coloring has two ways to output color 0, both 
through the satisfaction of $\tau_0$ and through the catch-all case. The part that forces 
color 0 to be present in all order-equivalent subsets is the satisfaction of $\tau_0$.
For arbitrary $H \oequiv \om^{d}\cdot k$ and $\tau$, we find variables $\rareconst_{\inumpoints}$ and $e_{\inumpoints,\idegr}$ such that
\begin{align*}
    \{\om^d \rareconst_1 + \om^{d-1} e_{1,d-1} + \cdots + \om^{1} e_{1,1} + e_{1,0}, \ldots,
    \om^d \rareconst_{\numpoints} + \om^{d-1} e_{\numpoints,d-1} + \cdots + \om e_{\numpoints,1} + e_{\numpoints,0}\}
\end{align*}
satisfies $\tau$.

For any $H \oequiv \om^{d}\cdot k$ and $\tau$, we first separate $H$ 
into $k$ ordered sets by the leading coefficient, each order-equivalent to $\om^d$.

Then, if there are equivalence classes in $\tau$, using the process formally described in the proof of Theorem~\ref{th:wd-GT},
we consider the leading equivalence class of $\tau$. By criterion~\ref{cr:y-distinct} 
of Definition~\ref{def:coloring-rule}, all variables in that equivalence class 
must come from the same set order-equivalent to $\om^d$. We assign a finite value to 
that equivalence class, and move to the next class
with a potentially different $\rareconst$ value
using the assigned finite value as a lower bound for the next one. We can repeat this process to find $e_{\inumpoints,\idegr}$ that satisfy every \crule of $\binom{\om^d \cdot k}{\numpoints}$ for arbitrary $H \oequiv \om^{d}$. Then, we can assign the $\rareconst_{\inumpoints}$ directly according to $\tau$.

If there are no equivalence classes in $\tau$ (it has size $p=0$), we can simply assign the variables $\rareconst_{\inumpoints}$ directly according to $\tau$.

Therefore for all $H \oequiv \omega^d$,
$\left|\COL\left(\binom{H}{\numpoints}\right)\right|\geq P(\numpoints,\om^d\cdot k)$. Hence $T(\numpoints,\om^d\cdot k) \geq P\left(\numpoints, \om^d\cdot k\right)$. Therefore by Theorem~\ref{th:wdk-LT}, $T(\numpoints,\om^d \cdot k) = P(\numpoints,\om^d \cdot k)$.
\end{proof}

\section{Big Ramsey Degrees of Ordinals less than \texorpdfstring{$\om^\om$}{omega omega}}\label{se:done}

\subsection{\capgencrules}
We defined \crules (\crulesLong) to compute big Ramsey degrees of ordinals of the form $\om^d \cdot k$. We
now extend the definition to \textit{\gencrulesnospace} (\gencrulesLong), which allows us to compute big Ramsey degrees for all ordinals less than $\om^\om$.

\begin{definition}
\label{def:SC} We now define \gencrules rigorously. Much like the definition of \crulesnospace, we impose a structure on edges, and then list criteria that \gencrules must satisfy.
Consider some ordinal $\alpha < \om^\om$:
$$\alpha \oequiv \om^d \cdot k_d + \om^{d-1}\cdot k_{d-1} + \cdots + \om \cdot k_1 + k_0.$$
Then $\alpha$ is the addition of $d+1$ ordinals each of the form $\om^\inumpoints\cdot k_\inumpoints$. 
For any element $\beta$ of $\alpha$, there is some $\inumpoints$ such that $$
\beta \in  \om^d \cdot k_d + \cdots + \om^{\inumpoints+1} \cdot k_{\inumpoints+1} + \om^{\inumpoints} \cdot k_{\inumpoints} \text{ and } \beta \not \in \om^d \cdot k_d + \cdots + \om^{\inumpoints+1} \cdot k_{\inumpoints+1}.
$$
In this case, we write that
$\beta$ {\it originated} from the $\om^{\inumpoints} \cdot k_{\inumpoints}$ part of $\alpha$. Note that the set of ordinals $\beta \in \alpha$ that originated from the $\om^{\inumpoints} \cdot k_{\inumpoints}$ part of $\alpha$ is order-equivalent to $\om^{\inumpoints} \cdot k_{\inumpoints}.$

For a natural number
$\numpoints\geq 0$, there are $\numpoints$ elements, denoted $p_\inumpoints \in \alpha,$
in any
$e \in \binom{\alpha}{\numpoints}$. Unlike in the definition of \crulesnospace, we will need to allow each $p_i$ to
have anywhere from 0 to $d$ variables $\commonconst_{\inumpoints,\idegr}$, depending on which of the $d+1$ ordered sets the element originated from. We use
$0 \leq c_{\inumpoints} \leq d$ for the 
number of such variables that the element $p_\inumpoints$ has.
Elements $p_\inumpoints \in \alpha$ with $c_{\inumpoints}$ variables
originated from the $\om^{c_{\inumpoints}} \cdot k_{c_{\inumpoints}}$ part of $\alpha$. We denote each element of $e$ as 
$$
\om^d \cdot k_d + \cdots + \om^{c_{\inumpoints}+1} \cdot k_{c_{\inumpoints}+1} + \om^{c_{\inumpoints}} \cdot \rareconst_{\inumpoints} + \om^{c_{\inumpoints}-1} \cdot \commonconst_{\inumpoints,c_{\inumpoints}-1} + \om^{c_{\inumpoints}-2} \cdot \commonconst_{\inumpoints,c_{\inumpoints}-2} + \cdots + \om^1 \cdot \commonconst_{\inumpoints,1} + \commonconst_{\inumpoints,0}.$$ where $0 \leq \rareconst_{\inumpoints} < k_{c_{\inumpoints}}$ and $0\leq \commonconst_{\inumpoints,\idegr}$. Note that for all elements originating from the $\om^{c_{\inumpoints}} \cdot k_{c_{\inumpoints}}$ part of $\alpha,$ the leading terms $$\om^d \cdot k_d + \cdots + \om^{c_{\inumpoints}+1} \cdot k_{c_{\inumpoints}+1}$$ are always the same; only the variables $c_{\inumpoints}, \rareconst_{\inumpoints},$ and $\commonconst_{\inumpoints,\idegr}$ are needed to uniquely identify an element.

A \textit{\gencruleLong}, hereafter referred to as a \gencrulenospace, is a triple $(\assnc,\assny, \xleq)$ of constraints on the $c_{\inumpoints}, \rareconst_{\inumpoints},$ and $\commonconst_{\inumpoints,\idegr}$.

The function $\assnc: [\numpoints] \to \{0,\dots,d\}$ maps indices $\inumpoints$ to their corresponding values $c_{\inumpoints}$. The function $\assny$ maps $\assny: [\numpoints] \to \omega$ such that $\assny(\inumpoints) < k_{\assnc(\inumpoints)},$ mapping from indices $\inumpoints$ to the values $\rareconst_{\inumpoints}$. The order $\xleq$ is a total preorder on the indices of the $\commonconst_{\inumpoints,\idegr}$. We continue to use the same notation to represent the preorder.

\gencrules must fulfill the following criteria (only criterion~\ref{cr:d-c-split} below is different from its corresponding criterion in Definition~\ref{def:coloring-rule}, the definition for \crulesnospace):
\begin{enumerate}
    \item \label{cr:d-index-ordering} If $d \geq 1$, then for all $\inumpoints < \inumpoints',$ $\commonconst_{\inumpoints, 0} < \commonconst_{\inumpoints', 0}.$ Otherwise when $d = 0$, $\rareconst_{\inumpoints} < \rareconst_{\inumpoints'}$ for all $\inumpoints<\inumpoints'$. (The element indices are ordered by their lowest-exponent variable.)
    \item \label{cr:d-y-distinct} $\commonconst_{\inumpoints,\idegr} = \commonconst_{\inumpoints',\idegr} \implies \rareconst_{\inumpoints} = \rareconst_{\inumpoints'}$ for all $\idegr$. (Elements with any $\commonconst$ values that are the same must have the same $\rareconst$ value.)
    \item \label{cr:d-high-first} For all $\idegr > \idegr',$ $\commonconst_{\inumpoints ,\idegr} < \commonconst_{\inumpoints, \idegr'}.$ (The high-exponent variables of each element are strictly less than the low-exponent variables.)
    
    \item \label{cr:d-term-dim-diff} $\commonconst_{\inumpoints, \idegr} = \commonconst_{\inumpoints' ,\idegr'} \implies \idegr = \idegr'$. (Only variables with the same exponent can be equal.)
    \item \label{cr:d-high-split} $\commonconst_{\inumpoints, \idegr} \neq \commonconst_{\inumpoints' ,\idegr} \implies \commonconst_{\inumpoints,\idegr-1} \neq \commonconst_{\inumpoints',\idegr-1}$ for all $\idegr>0$. (Elements that differ in a high-exponent variable differ in all lower-exponent variables.)
    \item \label{cr:d-c-split} $c_{\inumpoints} \neq c_{\inumpoints'} \implies (\rareconst_{\inumpoints} \neq \rareconst_{\inumpoints'}$ and $\commonconst_{\inumpoints ,\idegr} \neq \commonconst_{\inumpoints', \idegr})$ for all $0 \leq \idegr < d$. (Different $c$ variables mean different $\rareconst$ and $\commonconst$ variables.)
\end{enumerate}
\end{definition}

\begin{definition}
    Two \gencrules are \textit{equivalent} whenever their constraints $\assnc, \assny, \xleq$ are
    exactly equal.
\end{definition}

\begin{definition}
We again define the \textit{size} of a \gencrule to be how many equivalence classes its $\commonconst$ variables form. A \gencrulenospace's size $p$ is still bounded above by $d \cdot \numpoints$.
\end{definition}

\begin{definition}~
\begin{enumerate}
\item 
$S_p(\numpoints,\alpha)$ is the number of \gencrules of size $p$ there are for $\binom{\alpha}{\numpoints}$. 
\item 
$S(\numpoints,\alpha)$ is the total number of \gencrules there are for $\binom{\alpha}{\numpoints}$ regardless of size. It can be calculated as $$
\sum_{p=0}^{\numpoints \cdot d} S_p(\numpoints, \alpha).
$$ 
\end{enumerate}
\end{definition}
We will show $T(\numpoints,\alpha) = S(\numpoints,\alpha)$.

\begin{definition}\label{def:gcrule-coef}
For an edge $$e=\{\om^{c_{\inumpoints}} \cdot \rareconst_{\inumpoints} + \om^{c_{\inumpoints}-1} \cdot \commonconst_{\inumpoints,c_{\inumpoints}-1} + \om^{c_{\inumpoints}-2} \cdot \commonconst_{\inumpoints,c_{\inumpoints}-2} + \cdots + \om^1 \cdot \commonconst_{\inumpoints,1} + \commonconst_{\inumpoints,0} \colon 1 \leq \inumpoints \leq \numpoints\},$$
$e$ \textit{satisfies} the \gencrule $(\assnc,\assny, \xleq)$ if $c_{\inumpoints} = \assnc(\inumpoints)$ and $\rareconst_{\inumpoints} = \assny(\inumpoints)$ for every $1 \leq \inumpoints\leq \numpoints$, and if $(\inumpoints,\idegr) \xleq (\inumpoints',\idegr') \iff \commonconst_{\inumpoints,\idegr}\leq \commonconst_{\inumpoints',\idegr'}$ for all $1 \leq \inumpoints,\inumpoints'\leq \numpoints$, $0 \leq \idegr < c_{\inumpoints},$ and $0 \leq \idegr' < c_{\inumpoints'}.$
Again, the coefficients of an edge satisfying some \gencrule is the set of its
$a_{i,j}$ terms.
\end{definition}

\begin{lemma}
\label{le:D-is-P}
For $\numpoints,d,k,p \in \N$, $S_p(\numpoints, \om^d \cdot k) = P_p(\numpoints, \om^d \cdot k)$.
\end{lemma}
\begin{proof}
Let $\alpha = \om^d \cdot k$. Suppose some $c_{\inumpoints} \neq d$. Then $\rareconst_{\inumpoints} < k_{c_{\inumpoints}}$ by Definition~\ref{def:SC}, so $\rareconst_{\inumpoints} < 0$, which is impossible. Thus $c_{\inumpoints} = d$. Then there are the same count of $\numpoints \cdot d$ variables  $\commonconst_{\inumpoints,\idegr}$ being permuted, the new criterion~\ref{cr:d-c-split} has no effect because all $c_{\inumpoints}$ are equal. Hence both are under the same restrictions so $S_p(\numpoints, \om^d \cdot k) = P_p(\numpoints, \om^d \cdot k)$.
\end{proof}

\begin{lemma}
    \label{le:D-recur}
    For all $\alpha < \om^\om$ with \begin{align*}	\alpha \oequiv \om^d \cdot k_d + \om^{d-1} \cdot k_{d-1} + \cdots + \om \cdot k_1 + k_0 \text{ with } k_d \neq 0 \text{ and } d > 0,~~~~~~~~~\\
        S_p\left(\numpoints, \alpha\right) = \sum_{j=0}^{\numpoints} \sum_{i=0}^p \binom{p}{i} P_i\left(j, \om^d \cdot k_d \right) S_{p-i}\left(\numpoints-j, \om^{d-1} \cdot k_{d-1} + \cdots + \om \cdot k_1 + k_0\right).
     \end{align*}
     When the conditions $k_d \neq 0$ and $d > 0$ cannot be satisfied, then $\alpha < \om$ and $S_p(\numpoints, k_0) = P_p(\numpoints, k_0)$.
\end{lemma}
\begin{proof}
When $k_0$ is the only nonzero $k$ term, Lemma~\ref{le:D-is-P} shows $S_p(\numpoints, \alpha) = P_p(\numpoints, \alpha)$. When $d \geq 1$, we describe a process of combining \crules with \gencrules to create \gencrules for $\binom{\alpha}{\numpoints}$.

Let $\numpoints \geq 0, p \geq 0$ be naturals. For some $\alpha < \om^\om$, let $0 \leq m \leq \numpoints$ and $0 \leq s \leq p$ be naturals. We create $$\binom{p}{s} P_s \left( m,\om^{d} \cdot k_d \right) S_{p-s} \left( \numpoints-m,\om^{d-1} \cdot k_{d-1} + \cdots + k_0\right)$$ \gencrulesnospace, with each \gencrule having $m$ elements from the $\om^d  \cdot k_d$ part of $\alpha$ and $\numpoints-m$ elements from parts with lower exponents.

Let $\tau_1$ represent one of the $P_s\left(m, \om^{d} \cdot k_d\right)$ \crules of $\binom{\om^{d}\cdot k_d}{m}$ of size $s$, and $\tau_2$ represent one of the $S_{p-s}\left(\numpoints-m, \om^{d-1} \cdot k_{d-1} + \cdots + k_0\right)$ \gencrules of $\binom{\om^{d-1} \cdot k_{d-1} + \cdots + k_0}{\numpoints-m}$ of size $p-s$. We change $\tau_1$ into a \gencrule by assigning it $c_{\inumpoints} = d$ for all $c_{\inumpoints}$.

Then we can combine each $\tau_1$ and $\tau_2$ to form $\binom{p}{s}$ unique new \gencrules of size $p$: Change the indices $\tau_2$ and permute the equivalence classes as in the proof of Lemma~\ref{le:P-recur}. Note that we do not insert a leading equivalence class -- this is because we do not need to increase the exponent or size of $\tau_1$.

We can keep each $c_{\inumpoints}$ and $\rareconst_{\inumpoints}$ value the same, and change their indices alongside the $\commonconst_{\inumpoints,\idegr}$ variables to ensure criterion~\ref{cr:d-index-ordering}.

Each \gencrule is unique by the $\tau_1$ and $\tau_2$ used to create it because the process is invertible: we can identify the elements originally from $\tau_1$ because they uniquely have $c_{\inumpoints}=d$.

We claim each \gencrule created by this process has the properties described by Definition~\ref{def:SC}: Because all $c_{\inumpoints}$ are equal for $\tau_1$, criterion~\ref{cr:d-c-split} is satisfied for the elements from $\tau_1$.
Criterion~\ref{cr:index-ordering} is satisfied by changing the indices of the variables.
The remaining criteria are satisfied because $\tau_1$ and $\tau_2$ satisfied them and their internal orders and equivalence classes were preserved in permuting the equivalence classes. Therefore this process does not overcount \gencrulesnospace.

We also claim that every \gencrule of $\binom{\alpha}{\numpoints}$ is counted by this process: each can be mapped to some $\tau_1$ and $\tau_2$ that create it by a similar argument to proving that the process creates unique \gencrules 2 paragraphs above.
\end{proof}

\subsection{\texorpdfstring{$T(\numpoints,\alpha)\leq S(\numpoints,\alpha)$}{T(\numpoints, alpha) <= S(\numpoints, alpha)}}
\begin{lemma}
\label{le:bend-nats-alph}
For all $\alpha < \om^\om$, $\numpoints \in \N$, and ordered subsets
$G \subseteq \om$ with $G \oequiv \om$, there exists some $H \subseteq \alpha$, $H \oequiv \alpha$ where for all $e \in \binom{H}{\numpoints}$, $e$ satisfies a \gencrule of $\binom{\alpha}{\numpoints}$ and each coefficient of $e$ (see Definition~\ref{def:gcrule-coef}) is contained in $G$.
\end{lemma}
\begin{proof}
Because $G \oequiv \om$, $G = \{\commonconst_0 < \commonconst_1 < \commonconst_2 <
\cdots\}$.
Let $\alpha \oequiv \om^d \cdot k_d + \om^{d-1} \cdot k_{d-1} + \cdots + \om \cdot k_1 + k_0$.

Partition $G$ into $d+1$ distinct copies: \begin{align*}
    G_0 &= \{a_i : i \equiv 0 \mod d+1\}\\
    G_1 &= \{a_i : i \equiv 1 \mod d+1\}\\
    &\vdots\\
    G_{d} &= \{a_i : i \equiv d \mod d+1\}.
\end{align*}
For $0 \leq \idegr \leq d$, apply Lemma~\ref{le:bend-nats-k} on $G_\idegr$ to yield some $H_\idegr \oequiv \om^{\idegr} \cdot k_{\idegr}$ where all $e \in \binom{H_{\idegr}}{n}$ satisfy a \crule for $\binom{\om^{\idegr} \cdot k_{\idegr}}{\numpoints}$.
We then sum the remaining portion of $\alpha$ with each $H_j$, aligning with the order of $\alpha$:
$$
H'_j = \{\om^d \cdot k_d + \cdots + \om^{j+1} \cdot k_{j+1} + \beta : \beta \in H_j\} \oequiv H_j.
$$
Then let
$$H = H'_d \cup H'_{d-1} \cup \cdots \cup H'_0,$$
ordered by ordinal comparison. Since every element of $H'_{i+1}$ is less than every element of $H'_i$, $H \oequiv \alpha$.



We claim that all $e \in \binom{H}{n}$ satisfy a \gencrulenospace. Since we separated $G$ into disjoint orders $G_{\idegr}$, each $H'_{\idegr}$ is disjoint from the others.
With each $e' \in \binom{H_j}{n}$
satisfying a \crulenospace of a lesser ordinal by Lemma~\ref{le:bend-nats-k},
this continues to each $e' \in \binom{H'_j}{n}$ with the same $b_i, a_{i,j}$.
Hence
Criteria~\ref{cr:d-index-ordering}, \ref{cr:d-y-distinct}, \ref{cr:d-high-first},
\ref{cr:d-term-dim-diff}, and \ref{cr:d-high-split} of Definition~\ref{def:SC} are automatically satisfied for
$e$: if some
criterion was not satisfied and the $H'_{\idegr}$ remained distinct, that would
mean problematic elements came from the same part, contradicting how each $e' \in \binom{H'_j}{n}$
satisfies a \crulenospace.
Criterion~\ref{cr:d-c-split} is satisfied by each of the
$H'_{\idegr}$ being distinct from each other.
The coefficients of all edges $e$ are contained in $G$ by the construction of $H$ from $G$.
\end{proof}

\begin{theorem}
    \label{th:alph-GT}
    For all $\alpha < \om^\om$,
        $T\left(\numpoints, \alpha\right) \leq S\left(\numpoints, \alpha\right).$
\end{theorem}
\begin{proof}
Let $\numcolors \in \N$ and let \begin{align*}
    \COL \colon \binom{\alpha}{\numpoints} \to [\numcolors]
\end{align*}
be a $\numcolors$-coloring of $\binom{\alpha}{\numpoints}$.

Enumerate the \gencrules of $\binom{\alpha}{\numpoints}$ from $\tau_0$ to $\tau_{S(\numpoints, \alpha)-1}$. The maximum size of any \gencrule of $\binom{\alpha}{\numpoints}$ is $\numpoints \cdot d$. For each $\tau_i$, let \begin{align*}
    f_i \colon \binom{\om}{\numpoints \cdot d} \to \binom{\alpha}{\numpoints}
\end{align*}
where if $\tau_i$ has size $p$, $f_i$ maps $X$ to the unique $e \in \binom{\alpha}{\numpoints}$
where $e$ satisfies $\tau_i$ and the $p$ equivalence classes of $e$ are made up of the $p$ least elements of $X$. For example, one \gencrule of $\binom{\om^2 + \om\cdot 8}{2}$ is $$
c_1=2,c_2=1,\rareconst_1=0,\rareconst_2=6,~\commonconst_{1, 1} < \commonconst_{2, 0} < \commonconst_{1, 0}.
$$ The corresponding $f_i$ would be $f_i \colon \binom{\om}{4} \to \binom{\om^2 + \om\cdot 8}{2}$ with \begin{align*}
f_i(x_1,x_2,x_3,x_4) = \{\om^2 \cdot 0 + \om \cdot x_1 + x_3, \om^2 + \om \cdot 6 + x_2\}
\end{align*} where $x_1 < x_2 < x_3 < x_4$.

Then, define $\COL' \colon \binom{\om}{\numpoints \cdot d} \to [\numcolors]^{S(\numpoints, \alpha)}$ with \begin{align*}
    X = (\COL(f_0(X)),\COL(f_1(X)),\ldots,\COL(f_{S(\numpoints, \alpha)-1}(X)))
\end{align*}
and apply Theorem~\ref{th:ramsey} to find some $G \oequiv \om$ where
\begin{align*}
\left | \COL' \left( \binom{G}{\numpoints \cdot d} \right)\right | = 1.
\end{align*}
Let $\Rareconst$ be the one color in $\COL'\left(\binom{G}{\numpoints\cdot d}\right)$.
Note that $\Rareconst$ is a tuple of $S(\numpoints, \alpha)$ colors.

Apply Lemma~\ref{le:bend-nats-alph} to find some
$H \oequiv \alpha$ with the properties listed in Lemma~\ref{le:bend-nats-alph}. Now we claim 
\begin{align*}
    \left | \COL \left( \binom{H}{\numpoints} \right)\right | \leq 
S(\numpoints, \alpha).
\end{align*}

By Lemma~\ref{le:bend-nats-alph}, each element $e \in \binom{H}{\numpoints}$ satisfies a \gencrule of $\binom{\alpha}{\numpoints}$. 
Then for any edge $e$, let $e$ satisfy $\tau_i$ of size $p \leq \numpoints \cdot d$.
Then take the $p$ unique values in $e$, and if necessary, insert any new larger nonnegative naturals from 
$G$ to form a set of $\numpoints \cdot d$ values; denote this $\Commonconst \in \binom{G}{\numpoints \cdot d}$. 
Then
$\COL'(\Commonconst) = \Rareconst$ so by the definition of $\COL'$, $\COL(e) \in \Rareconst$. 
Because $|\Rareconst| = S(\numpoints, \alpha)$, $T(\numpoints,\alpha) \leq S(\numpoints, \alpha)$.
\end{proof}

\subsection{\texorpdfstring{$T(\numpoints,\alpha)\geq S(\numpoints,\alpha)$}{T(\numpoints, alpha) >= S(\numpoints, alpha)}}
\begin{theorem}
    \label{th:alph-LT}
    For all $\alpha < \om^\om, T\left(\numpoints, \alpha\right) \geq S\left(\numpoints, \alpha\right),$
    and thus by Theorem~\ref{th:alph-GT}, $T\left(\numpoints, \alpha\right) = S\left(\numpoints, \alpha\right).$
\end{theorem}
\begin{proof}
If $S(\numpoints, \alpha) = 0$, this is satisfied vacuously because $T(\numpoints, \alpha) \geq 0$. Suppose $S(\numpoints,\alpha) \geq 1$.
Note that all \gencrules of $\binom{\alpha}{\numpoints}$ are disjoint from each other. That is, for any edge $e \in \binom{\alpha}{\numpoints}$, if $e$ satisfies $\tau'$, then it does not satisfy any nonequivalent \gencrule of $\binom{\alpha}{\numpoints}$. This is because if $e$ were to satisfy two \gencrule $\tau_1$ and $\tau_2$, then $\tau_1$ and $\tau_2$ must share the same $\assnc, \assny$ functions and order $\xleq$, so the \gencrules must be equivalent. Therefore, we can index them $\tau_0, \ldots, \tau_{S(\numpoints, \alpha)-1}$ and construct a coloring $\COL \colon \binom{\alpha}{\numpoints} \to [S(\numpoints, \alpha)]$ with \begin{align*}
    \COL(e) = \begin{cases}
    i & e \textnormal{ satisfies } \tau_i\\
    0 & \textnormal{otherwise}
    \end{cases}
\end{align*}

For arbitrary $H \oequiv \alpha$ and a \gencrule $\tau$ for $\alpha$, we can assign $c_{\inumpoints}$ and $\rareconst_{\inumpoints}$ based on $\tau$. Then we can apply a similar process to the one used in Theorem~\ref{th:wdk-GT} to find $e_{\inumpoints,\idegr}$ variables that match the permutation of $\commonconst_{\inumpoints,\idegr}$ variables.

Let $\alpha \oequiv \om^d \cdot k_d + \om^{d-1} \cdot k_{d-1} + \cdots + \om \cdot k_1 + k_0$.
We can separate $H$ into $d+1$ sets each order-equivalent to $\om^{\idegr} \cdot k_{\idegr}$ for $0 \leq \idegr \leq d$, and separate each of those into $k_{\idegr}$ sets order-equivalent to $\om^{\idegr}$.

Then, for each equivalence class in $\tau$, using the process formally described in the proof of Theorem~\ref{th:wd-GT},
we consider the leading equivalence class of $\tau$. By criteria~\ref{cr:d-y-distinct} and~\ref{cr:d-c-split}
of Definition~\ref{def:SC}, all variables in that equivalence class 
must come from the same set order-equivalent to $\om^{\idegr}$. We assign a finite value to 
that equivalence class and move to the next class with potentially different 
$c$ and $\rareconst$ values, using the assigned finite value as a lower bound for the next one. We can repeat this process to find $e_{\inumpoints ,\idegr}$ that satisfy each \gencrule of $\binom{\alpha}{\numpoints}$ for arbitrary $H \oequiv \alpha$.

Therefore for all 
$H \oequiv \alpha$,
$\left|\COL\left(\binom{H}{\numpoints}\right)\right|\geq S(\numpoints,\alpha)$ so $T(\numpoints,\alpha) \geq S\left(\numpoints, \alpha\right)$.
With Theorem~\ref{th:alph-GT}, we have
$T(\numpoints,\alpha) = S\left(\numpoints, \alpha\right)$.
\end{proof}





\subsection*{Acknowledgments}

We would like to thank Natasha Dobrinen for introducing us to this
subject, giving us advice on the project, and helpful comments on the 
final draft. We thank Robert Brady and Erik Metz for the helpful comments on the final draft.
We would also like to thank 
Nathan Cho, 
Isaac Mammel, and 
Adam Melrod for helping us simplify the proof of Theorem~\ref{th:zeta}.

\section*{Appendix}
Table~\ref{tab:om-d-a} shows $T(\numpoints,\om^d)$ for small $\numpoints,d$.
\begin{table}[htpb]
\centering
\begin{tabular}{| c | c | r | r | r | r | r | r|}
    \hline
     \multicolumn{2}{|c|}{\multirow{2}{*}{$T(n,\om^d)$}}&\multicolumn{6}{c|}{$\numpoints$}\\\cline{3-8}
     \multicolumn{1}{|c}{}&&0&1&2&3&4&5\\\hline
     \multirow{6}{*}{$d$}&0&1&1&0&0&0&0\\\cline{2-8}
     &1&1&1&1&1&1&1\\\cline{2-8}
     &2&1&1&4&26&236&2752\\\cline{2-8}
     &3&1&1&14&509&35839&4154652\\\cline{2-8}
     &4&1&1&49&10340&5941404&7244337796\\\cline{2-8}
     &5&1&1&175&222244&1081112575&14372713082763\\\hline
\end{tabular}
\caption{Big Ramsey degrees of $\om^d$.}
\label{tab:om-d-a}
\end{table}

\end{document}